\documentclass[10pt,leqno,a4paper]{article}
\usepackage{a4wide}
\usepackage{color}
\usepackage{amsfonts}
\usepackage{amsmath, amsfonts, amsthm, amssymb, amscd}
\usepackage[mathcal]{euscript}
\usepackage{mathrsfs}
\usepackage{dsfont}
\usepackage{ulem}
\usepackage{bbm}
\usepackage{graphicx}

\usepackage{mathtools}
\usepackage{slashed}

\newtheorem{theorem}{Theorem}[section]
\newtheorem{proposition}[theorem]{Proposition}
\newtheorem{lemma}[theorem]{Lemma}

\newtheorem{remark}[theorem]{Remark}

\numberwithin{equation}{section}

\usepackage{multirow}
\usepackage{tabularx}
\usepackage{lscape}

\newcommand \Kcal {\mathcal K}

\newcommand \Fcal{\mathcal{F}}
\newcommand \Hcal {\mathcal H}

\newcommand \Lcal {\mathcal L}

\newcommand \Ecal {\mathcal E}

\newcommand \delb {\bar {\del}}

\newcommand \phit {\tilde{\phi}}

\newcommand \m{\text{m}}

\newcommand \del \partial
\newcommand \delu {\uline{\del}}

\newcommand \Tu {\uline{T}}

\newcommand \Psiu{\uline{\Psi}}
\newcommand \Phiu{\uline{\Phi}}

\newcommand \Ec {E_2}

\newcommand \RR{\mathbb{R}}

\newcommand {\vep}{\varepsilon}

\newcommand {\dels}{\slashed{\del}}

\makeatletter
\def\hlinew#1{%
  \noalign{\ifnum0=`}\fi\hrule \@height #1 \futurelet
   \reserved@a\@xhline}
\makeatother

\title{Global solutions of wave-Klein-Gordon system in two spatial dimensions with strong couplings in divergence form \footnote{The present work belongs to a research project ``Global stability of quasilinear wave-Klein-Gordon system in $2 + 1$ space-time dimension'' (11601414), supported by NSFC.}}

\author{Senhao Duan \& Yue Ma \footnote{School of Mathematics and Statistics, Xi'an Jiaotong University, Xi'an, Shaanxi 710049, P.R. China.\ E-mail: yuemath@xjtu.edu.cn}}

\begin{document}

\maketitle
\begin{abstract}
In this paper we established the global well-posedness theorem for a special type of wave-Klein-Gordon system that have the strong coupling terms in divergence form on the right hand side of its wave equation. We cope with the problem by constructing an auxiliary system with the shifted primitives of the original unknowns. The result is then applied directly on Klein-Gordon-Zakharov system in $2+1$ space-time with general  small-localized-regular initial data. In the end of this paper, we also give a preliminary answer to the question of global stability of a class of totally geodesic the wave maps in 2+1 dimensional case.
\end{abstract}
\section{Introduction}
\subsection{Objective and main result}
This article belongs to a research project in which we attempt to understand the effects of different quadratic terms coupled in diagonalized wave-Klein-Gordon system in $2+1$ dimensional space-time. In this article, we are interested in a special type of wave-Klein-Gordon system represented by the following two systems:
\begin{subequations}\label{eq1-main}	
\begin{equation}\label{eq1a-main}
\aligned
&\Box u = A^{\alpha}\del_{\alpha}(v^2),
\\
&\Box v + c^2v = B^{\alpha}v\del_{\alpha}u;
\endaligned
\end{equation} 	
\begin{equation}\label{eq1b-main}
\aligned
&\Box u = A^{\alpha\beta}\del_{\alpha}\del_{\beta}(v^2),
\\
&\Box v + c^2v = Buv.
\endaligned
\end{equation}

\end{subequations}
It can be noticed that on the right-hand-side of the wave equations, is the strong coupling terms term introduced in \cite{M-2020-strong}. We will establish global existence results for these systems with small localized regular initial data, more precisely,
\begin{theorem}\label{thm-main}
	Consider the Cauchy problems associated to \eqref{eq1-main} with initial data posed on $\{t=2\}$ and compactly supported in $\{|x|<1\}$:
	\begin{align*}
	&v(2,x) = v_0(x),\quad \del_tv(2,x) = v_1(x)
	\\
	&u(2,x) = u_0(x),\quad \del_tu(2,x) = u_1(x).
	\end{align*}
	There exists an integer $N\geq 9$ and a positive constant $\vep_0>0$ determined by the system and $N$, such that for all $0\leq \vep\leq \vep_0$, if
	\begin{equation}
	\|u_0\|_{H^{N+1}} + \|v_0\|_{H^{N+1}} + \|u_1\|_{L^2(H^N)} + \|v_1\|_{H^N}\leq \vep,
	\end{equation}
	then the local-in-time solution of \eqref{eq1-main} associated with such initial data extends to time infinity.
\end{theorem} 

The research on \eqref{eq1a-main} is motivated by a stability problem of a type of totally geodesic wave map. In \cite{Ab-2019} the following system was formulated:
\begin{equation}\label{eq1-wave-map}
\begin{aligned}
&-\Box \phi^1 = -2\sum_{k=2}^n\phi^k \del_1\phi^k + \text{h.o.t}
\\
&-\Box \phi^k - \phi^k = 2\phi^k \del_1u + \text{h.o.t.}, \quad k=2,\cdots, n
\end{aligned}
\end{equation}
where $u$ and $\phi^k$ are scalar functions defined in $\RR^{2+1}$. Relied on this formulation, a global stability result on wave map in $3+1$ and higher dimension was established in \cite{Ab-2019}. The cases in lower dimension was suggested to be open problems therein. In this article we will give a preliminary answer to the $2+1$ dimensional case (Theorem \ref{thm-wave-map}). In Section \ref{sec-conclusion-wave-map} we sketch the geometric background of \eqref{eq1-wave-map}. Detailed discussions on the formulation of \eqref{eq1-wave-map} can be found in \cite{Ab-2019} and for general review on wave maps,  one may read \cite{SS98} and \cite{Kri07}. 

The research on \eqref{eq1b-main} is motivated by the global stability problem of Klein-Gordon-Zakharov system:
\begin{equation}\label{eq1-Zakharov}
\aligned
& \Box E^a + E^a = -nE^a,\quad a = 1,2,
\\
& \Box n = \Delta \big(|E^1|^2 + |E^2|^2\big),
\endaligned
\end{equation}
where $n, E^a$ are scalar functions. The Zakharov equation was introduced in \cite{Zakharov-1972}. It describes a type of oscillation of a plasma. The Klein-Gordon-Zakharov system is a typical wave-Klein-Gordon system. The global stability result in $3+1$ space-time was established in \cite{Ozawa-1995} with Fourier-analytic method and latter in \cite{Tsutaya-1996} via vector field method. This result is then revisited and improved in many context. The main challenge of regarding wave-Klein-Gordon system comes form the lack of scaling invariance of the Klein-Gordon equation. See \cite{LM1} for a detailed explanation. 

Recently, S. Dong \cite{Dong-2020-2} established the global stability result in $2+1$ space-time with a special type of initial data. More precisely, Dong's method shows that, if there exists a compactly supported function $n^{\Delta}$ such that $\Delta n = n^{\Delta}$ on the initial slice, then with suitable assumptions on the regularity and smallness of the initial, the associated local-in-time solution extends to time infinity.  His method is based on hyperboloidal foliation combined with a weighted energy estimate (called ``ghost weight'').

In this article, as we have showed in the statement of Theorem \ref{thm-main}, we managed to establish a global stability result for general initial data in the small-localized-regular regime.

\subsection{Main difficulties and strategy of proof}
As explained in \cite{M-2020-strong}, in $\RR^{2+1}$, the main difficulty concerning the strong coupling terms, i.e., pure Klein-Gordon quadratics in wave equation, is that they destroy ``completely'' the conformal invariance of the wave equation (which supplies better decay and energy bounds). It seems to be impossible to establish uniform or slowly increasing conformal energy bound on wave component. Then one will face the insufficiency of the so-called principle decay.
See in \cite{M-2020-strong} for a detailed explanation. Roughly speaking, in the case of strong coupling, one can only expect uniform standard energy bound. This bound leads to (via Klainerman-Sobolev inequality) the following decay
\begin{equation}\label{eq3-04-10-2020}
|\del u|\simeq s^{-1}\simeq (|t-r|+1)^{-1/2}t^{-1/2}
\end{equation}
which will not be sufficient to close the bootstrap argument.

Fortunately, in the present case the strong couplings are in divergence form. This motivates us to ``integrate'' the wave equation, i.e, regarding the ``primitives'' of the wave component instead of it-self. The advantage of this strategy is that, the primitives also satisfy a wave equation (again with strong couplings), and the wave component is regarded as derivative of these primitives. Then the gradient of the wave component coupled in Klein-Gordon equation becomes components of Hessian forms of the primitives. As explained in \cite{M-2020-strong} (see also Proposition \ref{prop1-14-08-2020} in detail), Hessian form of a solution to wave equation enjoys better decay and energy bounds in the sens of principle decay and this will bring us not a little convenience. Here we only show an example. Compared with \eqref{eq3-04-10-2020}, when the standard energy on hyperboloid is uniformly bounded,
\begin{equation}\label{eq4-04-10-2020}
|\del\del u|\simeq (s/t)^{-1}s^{-2} + (s/t)^{-1}|\Box u| \simeq (|t-r|+1)^{-3/2}t^{-1/2} + t^{1/2}(1+|t-r|)^{-1/2}|\Box u|
\end{equation}
where $|\Box u|$ is quadratic by applying the wave equation and can be expected to have sufficient decay. Comparing \eqref{eq4-04-10-2020} with \eqref{eq3-04-10-2020}, the improvement only occurs deep in the right-cone $\{r<(1-\delta)t\}$. However, this is already sufficient in order to get integrable $L^2$ bounds on $v\del\del u$. More precisely, the Klein-Gordon component enjoys fast conical bounds:
$$
\|(s/t)^{-2}v\|_{L^2(\Hcal_s)}\lesssim \Ecal^N(s,v)^{1/2}
$$ 
for $N$ sufficiently large (see in detail in the proof or observe it roughly via Proposition \ref{prop1-fast-kg}). This additional $(s/t)^{-2}$ weight offsets the $(s/t)^{-1}$ conical decay in \eqref{eq4-04-10-2020}. Then (roughly) one obtains
\begin{equation}
\|v\del\del u\|_{L^2(\Hcal_s)}\lesssim s^{-2}\|(s/t)|v|\|_{L^2(\Hcal_s)}\lesssim s^{-2}\Ecal^N(s,v)^{1/2}
\end{equation} 
with $s^{-2}$ integrable with respect to $s$.  With this observation on divergence $\rightarrow$ primitive $\rightarrow$ Hessian form, we will be able to treat some originally non-integrable quadratic terms. 

However, writing the system with primitives is not a gratis trick.  As we will see in the following analysis, although  a primitive of wave component also resolves a wave equation, the initial data can not be easily constructed. To overcome this we consider a ``modified'' primitive instead, which is the primitive shifted by a solution to a free-linear wave equation. In Section \ref{sec-reformulation}, the system \eqref{eq1-main} will be reformulated with these shifted primitives and this leads to an auxiliary systems in the form of \eqref{eq-main}. In subsection \ref{subsec-auxi-structure} we will give a more detailed investigation on the structure of this type of system.

The present article is roughly compose by three parts. Section \ref{sec-tech} forms the first part in which we prepare all analytical tools. The second part composed by Section \ref{sec-reformulation} and \ref{sec-bootstrap} in which we establish global existence result on \eqref{eq1-main} and apply this on \eqref{eq1-Zakharov}. The last part, containing Section \ref{sec-conclusion-wave-map} and \ref{sec-wave-maps-proof}, is dedicated to the stability result of totally geodesic wave map in which we regard the full system formulated in \cite{Ab-2019}. The proof is quite similar to that of \eqref{eq1-main} in Section \ref{sec-bootstrap}. But due to the higher-order terms and some other structures, neither can be seen as a special case of the other.


\section{Recall of some technical tools}\label{sec-tech}
In this section we are going to recall some useful tools in the hyperboloidal foliation method. We will start with the basic notations of the frames, vector fields and the high-order derivatives in the first subsection. Then we recall / reformulate some estimates based on the linear structure of wave / Klein-Gordon equation in the following two subsections.

\subsection{Basic notation and calculus within hyperboloidal frame work}
\paragraph*{Frames and vector fields.}
We are interested in the foliation of the interior of the light cone $\Kcal \coloneqq \{(t,x)|r < t-1\}\subset \RR^{2+1}$ where $(t,x)=(t,x^a)=(t,x^1,x^2)$ is the Cartesian coordinates and $r = \sqrt{|x^1|^2+|x^2|^2}$. Then the foliation is performed with $\Hcal_s \coloneqq \{(t,x)| t = \sqrt{s^2+r^2}\}$ as following:
\begin{align*}
\Hcal_{[s_0,s_1]} &\coloneqq \bigcup_{s_0\leq s\leq s_1}(\Hcal_s\cap \Kcal)
                  =\{(t,x)|r< t-1, (s_0)^2\leq s^2 \leq (s_1)^2\},
\end{align*}
and
\begin{align*}
\Hcal_{[s_0,\infty]} &\coloneqq \bigcup_{s\geq s_0}(\Hcal_s\cap \Kcal)
                  =\{(t,x)|r< t-1, s^2\geq(s_0)^2 \}.
\end{align*}
We recall the semi-hyperboloidal frame introduced in \cite{LM1} \footnote{Throughout this article, Greek indices taking values in $\{0,1,2\}$ while Latin indices taking values in $\{1,2\}$.}:
$$
\delu_0:=\del_t,\quad \delu_a := \delb_a = (x^a/t)\del_t + \del_a,
$$
where $\delb_a$ denotes the vector fields tangent to the hyperboloids $\Hcal_s$ (which are called hyperbolic derivatives).
By a direct computation, we have the transition matrices between this frame and the natural frame $\{\del_{\alpha}\}$ as follows:
\begin{equation}\label{eq semi-frame}
\Phiu_{\alpha}^{\beta} := \left(
\begin{array}{ccc}
1 &0 &0
\\
x^1/t &1 &0
\\
x^2/t &0 &1 
\end{array}
\right),
\quad
\Psiu_{\alpha}^{\beta} := \left(
\begin{array}{ccc}
1 &0 &0
\\
-x^1/t &1 &0
\\
-x^2/t &0 &1 
\end{array}
\right)
\end{equation}
with
$$
\delu_{\alpha} = \Phiu_{\alpha}^{\beta}\del_{\beta},\quad \del_{\alpha} = \Psiu_{\alpha}^{\beta}\delu_{\beta}.
$$

Hence, assume that $T = T^{\alpha\beta}\del_{\alpha}\otimes\del_{\beta}$ be any 2-tensor defined in $\Kcal$ or its subset, it can be also represented by $\{\delu_{\alpha}\}$ as following:
$$
T = \Tu^{\alpha\beta} \delu_{\alpha}\otimes\delu_{\beta}
\quad\text{with}\quad 
\Tu^{\alpha\beta} = T^{\alpha'\beta'}\Psiu_{\alpha'}^{\alpha}\Psiu_{\beta'}^{\beta}.
$$

\paragraph*{High-order derivatives.}
Recall that in the region $\Kcal$, we introduced the following Lorentzian boosts in \cite{M-2020-strong}:
$$
L_a = x^a\del_t + t\del_a,\quad a = 1, 2
$$
and the following notation of high-order derivatives: let $I= (i_1,i_2,\cdots, i_m)$, $J= (j_1,j_2,\cdots, j_n)$ be multi-indices taking values in $\{0,1,2\}$ and $\{1,2\}$ respectively,
and then we define
$$
\del^IL^J = \del_{i_1}\del_{i_2}\cdots \del_{i_m}L_{j_1}L_{j_2}\cdots L_{j_n}
$$
to be an $(m+n)-$order derivative. 

Let $\mathscr{Z} = \{Z_i|i=0,1,\cdots, 6\}$ be a fimily of vector fields, where
$$
Z_0 = \del_t,\quad  Z_1=\del_1,\quad Z_2 = \del_2,\quad Z_3 = L_1,\quad Z_4=L_2,\quad Z_5 = \delu_1,\quad Z_6 = \delu_2.
$$
The following notation:
$$
Z^I := Z_{i_1}Z_{i_2}\cdots Z_{i_N}
$$
denotes a high-order derivative of order $N$ on $\mathscr{Z}$ with milti-index $I = (i_1,i_2,\cdots, i_N)$ with $i_k\in \{1,2,c\dots, 6\}$. If there are at most $a$ partial derivatives, $b$ Lorentzian boosts and $c$ hyperbolic derivatives in $Z^I$, then $I$ is said to be of type $(a,b,c)$.

We then recall the following notation introduced in \cite{M-2020-strong}:
$$
\mathcal{I}_{p,k} = \{I| I \text{ is of type }(a,b,0)\text{ with }a+b \leq p, b\leq k \},
$$
and the following quantities that will be applied in order to control varies of high-order derivatives later:

\begin{equation}\label{eq1 notation}
\aligned
|u|_{p,k} &:= \max_{K\in \mathcal{I}_{p,k}}|Z^K u|,\quad &&|u|_p := \max_{0\leq k\leq p}|u|_{p,k},
\\
|\del u|_{p,k} &:= \max_{\alpha=0,1,2}|\del_{\alpha} u|_{p,k}, &&|\del u|_p := \max_{0\leq k\leq p}|\del u|_{p,k},
\\
|\del^m u|_{p,k} &:= \max_{|I|=m}|\del^I u|_{p,k}, &&|\del^m u|_p := \max_{0\leq k\leq p}|\del^I u|_{p,k},
\\
|\dels u|_{p,k} &:= \max\{|\delu_1 u|_{p,k},|\delu_2u|_{p,k}\}, &&|\dels u|_p := \max_{0\leq k\leq p}|\dels u|_{p,k},
\\
|\del\dels  u|_{p,k} &:=\max_{a,\alpha} \{|\delu_a\del_{\alpha} u|_{p,k},|\del_{\alpha}\delu_a u|_{p,k}\},
&&| \del\dels u|_p :=\max_{0\leq k\leq p}| \del\dels u|_{p,k},
\\
|\dels\dels  u|_{p,k} &:=\max_{a,b} \{|\delu_a\delu_b u|_{p,k}\},
&&| \dels\dels u|_p :=\max_{0\leq k\leq p}| \dels\dels u|_{p,k}.
\endaligned
\end{equation}

\paragraph*{Standard and Conformal energy estimate on hyperboloids.}
There are two types, standard and conformal, of the energies defined in the hyperboloidal foliation framework. The standard energy, obtained by the standard multiplier $\del_t u$, is defined as follows in the Minkowski metric: 
\begin{equation}\label{standard energy}
E_{0,c}(s,u):= \int_{\Hcal_s}e_{0,c}[u]dx
\end{equation}
where the energy density
\begin{equation}\label{density-standard}
\aligned
e_{0,c}[u]:=&|\del_tu|^2+\sum_a|\del_au|^2 + 2(x^a/t)\del_tu\del_au + c^2u^2 
\\
=&\sum_a |\delu_a u|^2 + |(s/t)\del_tu|^2 + c^2u^2
\\
=&|\delu_{\perp}u|^2 + \sum_a|(s/t)\del_a u|^2 + \sum_{a<b}\big|t^{-1}\Omega_{ab}u\big|^2 + c^2u^2
\endaligned
\end{equation}
with $\delu_{\perp} := \del_t + (x^a/t)\del_a$. We denote by $e_0[u] = e_{0,c=0}[u]$.

For standard energy, we have the following estimate (for proof, see for example \cite{LM1}):
\begin{proposition}[Standard energy estimate]\label{prop 1 energy}
	We consider the $C^2$ solution $u$ to the following wave / Klein-Gordon equation
	$$
	\Box u + c^2 u = F,
	$$
	in the region $\Hcal_{[s_0,s_1]}$ and vanishes near the conical boundary $\del\Kcal = \{t=r-1\}$. Then the following energy estimate holds:
	\begin{equation}\label{ineq 3 prop 1 energy}
	\aligned
	E_{0,c}(s,u)^{1/2}\leq& E_{0,c}(2,u)^{1/2} + \int_2^s \|F\|_{L^2(\Hcal_\tau)}d\tau.
	\endaligned
	\end{equation}
\end{proposition}

While the conformal energy on hyperboloid $\Hcal_s$ is defined as follows:
$$
E_2(s,u) := \int_{\Hcal_s}\Big(\sum_a|s\delu_a u|^2 + s^{-2}|K_2u + su|^2\Big)dx
$$
where $K_2 = s^2(s/t)\del_t + 2sx^a\delu_a$ is the conformal multiplier. We also have an estimate for the energy of this type:

\begin{proposition}[Conformal energy estimate on hyperboloids]\label{prop-conformal}
	Let $u$ be a sufficiently regular function defined in $\Hcal_{[s_0,s_1]}$, vanishes near the conical boundary $\del\Kcal = \{r=t-1\}$. Then the following estimate holds:
	\begin{equation}\label{eq8-15-06-2020}
	\Ec(s_1,u)^{1/2} \leq \Ec(s_0,u)^{1/2} + \int_{s_0}^{s_1}s\|\Box u\|_{L^2(\Hcal_s)}ds.
	\end{equation}
\end{proposition}
\noindent Unlike the standard energy, the conformal one does not directly control the derivative $\del_t u$ and $u$. Therefore, the following lemma is established in \cite{M4} in order to get the bound on $u$: 

\begin{lemma}\label{proposition 1 01-01-2019}
	Let $u$ be a $C^1$ function defined in $\Hcal_{[s_0,s_1]}$ and vanishes near $\del\Kcal$. Then
	\begin{equation}\label{eq 1 14-12-2018}
	\|(s/t)u\|_{L^2(\Hcal_{s_1})}\leq \|(s/t)u\|_{L^2(\Hcal_{s_0})} + C\int_{s_0}^{s_1}s^{-1}\Ec(s,u)^{1/2}ds.
	\end{equation}
\end{lemma}
Once $u$ is bounded, recalling 
$$
\|s^{-1}K_2 u + u\|_{L^2(\Hcal_s)} = \|s(s/t)\del_t u + 2x^a\delu_a u\|_{L^2(\Hcal_s)}
$$ is bounded by $E_2(s,u)^{1/2}$, $\|(s/t)^2s\del_t u\|_{L^2(\Hcal_s)}$ is bounded by the following quantity:
\begin{equation}\label{eq1-09-10-2020}
F_2(s_0,s,u) = \|(s/t)u\|_{L^2(\Hcal_s)} + E_2(s,u)^{1/2} + \int_{s_0}^s\tau^{-1}E_2(\tau,u)^{1/2}d\tau
\end{equation}
The high-order version is defined as following:
\begin{equation}
\Fcal_2^{p,k}(s_0;s,u):= \max_{|I|+||J\leq p\atop |J|\leq k}F_2(s_0;s,\del^IL^J u),\quad
\Fcal_2^N(s_0;s,u):= \max_{|I|+|J|\leq N}F_2(s_0;s,\del^IL^J u).
\end{equation}
A sketch on the proofs of this conformal energy estimate and Lemma \ref{proposition 1 01-01-2019} within flat background metric can be found in \cite{M-2020-strong}. 

\noindent For the convenience of discussion, we also introduce the following high-order energy:

\begin{equation}\label{eq 3' 01-01-2019}
\Ecal_{0,c}^{p,k}(s,u) := \max_{|I|+J\leq p\atop |J|\leq k}E_{0,c}(s,\del^IL^J u),\quad \Ecal_0^{p,k}(s,u) := \max_{|I|+|J|\leq p\atop |J|\leq k}E_0(s,\del^IL^J u),
\end{equation}
\begin{equation}\label{eq 3 01-01-2019}
\Ecal_{0,c}^N(s,u) := \max_{|I|+|J|\leq N}E_{0,c}(s,\del^IL^J u),\quad \Ecal_0^N(s,u) := \max_{|I|+|J|\leq N}E_0(s,\del^IL^J u),
\end{equation}
\begin{equation}\label{eq 1 09-26-2020}
\Ecal_2^N(s,u) := \max_{|I|+|J|\leq N}E_2(s,\del^IL^J u), \quad \Ecal_2^{p,k}(s,u) := \max_{|I|+|J|\leq p\atop |J|\leq k}E_2(s,\del^IL^J u).
\end{equation}

\paragraph*{Bounds of high-order derivatives with energies.}
These bounds are established in \cite{M-2020-strong}:
\\
- $L^2$ bounds:
\begin{equation}\label{eq1-10-06-2020}
\|(s/t)|\del u|_{p,k}\|_{L^2(\Hcal_s)} + \||\dels u|_{p,k}\|_{L^2(\Hcal_s)} 
+ \|c|u|_{p,k}\|\leq C\Ecal_{0,c}^{p,k}(s,u)^{1/2},
\end{equation}
\begin{equation}\label{eq5-10-06-2020}
\|s|\del\dels u|_{p-1,k-1}\|_{L^2(\Hcal_s)} + \|t|\dels\dels u|_{p-1,k-1}\|_{L^2(\Hcal_s)} \leq C\Ecal_0^{p,k}(s,u)^{1/2},
\end{equation}
\begin{equation}\label{eq2-10-06-2020}
\aligned
\|(s/t)^2s|\del u|_{p,k}\|_{L^2(\Hcal_s)}& + \|s|\dels u|_{p,k}\|_{L^2(\Hcal_s)} 
+  \|(s/t)|u|_{p,k}\|_{L^2(\Hcal_s)}
\\
\leq& C\Fcal_2^{p,k}(s_0;s,u),
\endaligned
\end{equation}
\begin{equation}\label{eq6-10-06-2020}
\aligned
\|(s/t)s^2|\del\dels u|_{p-1,k-1}\|_{L^2(\Hcal_s)}& + \|st|\dels\dels u|_{p-1,k-1}\|_{L^2(\Hcal_s)}
\\
\leq& C\Fcal_2^{p,k}(s_0;s,u).
\endaligned
\end{equation}
\\
- $L^{\infty}$ bounds:
\begin{equation}\label{eq3-10-06-2020}
\aligned
\|s|\del u|_{p,k}\|_{L^{\infty}(\Hcal_s)}& + \|t|\dels u|_{p,k}\|_{L^\infty(\Hcal_s)}  + \|ct|u|_{p,k}\|_{L^{\infty}(\Hcal_s)}
\\
\leq& C\Ecal_{0,c}^{p+2,k+2}(s,u)^{1/2},
\endaligned
\end{equation}
\begin{equation}\label{eq7-10-06-2020}
\|st|\del\dels u|_{p-1,k-1}\|_{L^\infty(\Hcal_s)} + \|t^2|\dels\dels u|_{p-1,k-1}\|_{L^{\infty}(\Hcal_s)} \leq C\Ecal_0^{p+2,k+2}(s,u)^{1/2},
\end{equation}
\begin{equation}\label{eq4-10-06-2020}
\aligned
\|(s/t)s^2|\del u|_{p,k}\|_{L^\infty(\Hcal_s)} + \|st|\dels u|_{p,k}\|_{L^\infty(\Hcal_s)} 
&+ \|s|u|_{p,k}\|_{L^\infty(\Hcal_s)}
\\
\leq& C\Fcal_2^{p+2,k+2}(s_0;s,u),
\endaligned
\end{equation}
\begin{equation}\label{eq8-10-06-2020}
\aligned
\|s^3|\del\dels u|_{p-1,k-1}\|_{L^{\infty}(\Hcal_s)}& + \|st^2|\dels\dels u|_{p-1,k-1}\|_{L^{\infty}(\Hcal_s)} 
\\
\leq& C\Fcal_2^{p+2,k+2}(s_0;s,u).
\endaligned
\end{equation}

We also need the following bound on products and null quadratic forms in $\Hcal_{[s_0,s_1]}$. Firstly,
\begin{equation}\label{eq12-10-06-2020}
|AB|_{p,k}\leq C|A|_{p,k}|B|_{p_1,k_1} + C|A|_{p_1,k_1}|B|_{p,k} 
\end{equation}
where $p_1 = [p/2], k_1 = [k/2]$, $A,B$ sufficiently regular in $\Hcal_{[s_0,s_1]}$ and $C$ a constant determined by $p$. Furthermore, let $A$ be a (constant coefficient) quadratic null form, i.e.,
$$
A^{\alpha\beta}\xi_{\alpha}\xi_{\beta} = 0,\quad \forall \xi_0^2 - \xi_1^2 - \xi_2^2 = 0.
$$
Then
\begin{equation}\label{eq13-10-06-2020}
\aligned
|A^{\alpha\beta}\del_{\alpha}u\del_{\beta}v|_{p,k}\leq& C|A|(s/t)^2|\del u|_{p_1,k_1}|\del v|_{p,k} + C(s/t)^2|A||\del u|_{p,k}|\del v|_{p_1,k_1} 
\\
&+ C|A||\dels u|_{p_1,k_1}|\del u|_{p,k} + C|A||\dels u|_{p,k}|\del v|_{p_1,k_1}
\\
&+C|A||\del u|_{p_1,k_1}|\dels v|_{p,k} + C|A||\del u|_{p,k}|\dels v|_{p_1,k_1}
\endaligned
\end{equation}
where $|A| = \max_{\alpha,\beta}|A^{\alpha\beta}|$. This is established in \cite{LM1}. For a proof, see for example in \cite{M-2020-strong}.

\subsection{Linear estimates on wave equation}
\paragraph*{Bounds on Hessian form of wave component.}
We are now at a state to recall various bounds of the wave and the Klein-Gordon equation due to their linear structure. For the Hessian form, we have the coming proposition:
\begin{proposition}\label{prop1-14-08-2020}
	Let $u$ be a function defined in $\Hcal_{[s_0,s_1]}$, sufficiently regular. Suppose that $|I|+|J|\leq p$ and $|J|\leq k$. Then 
	\begin{equation}\label{eq1 lem Hessian-flat-zero}
	(s/t)^2|\del_\alpha\del_\beta \del^IL^J u| \leq C|\Box u|_{p,k} + Ct^{-1}|\del u|_{p+1,k+1}.
	\end{equation}
	\begin{equation}\label{eq2 lem Hessian-flat-zero}
	(s/t)^2|\del\del  u|_{p,k} \leq C|\Box u|_{p,k} + Ct^{-1}|\del u|_{p+1,k+1}.
	\end{equation}
\end{proposition}
This is established in \cite{LM1}. A sketch of proof can be found in \cite{M-2020-strong}.

\paragraph*{Decay bounds based on Poisson's formula.}
By a direct calculation with the Poisson's formula, we have the following decay bounds on the free-linear wave equation:
\begin{lemma}\label{lem1-13-06-2020}
	Let $u$ be the $C^2$ solution to the following Cauchy problem of free-linear wave equation:
	\begin{equation}
	\Box u = 0,\quad u(t_0,x) = u_0, \quad \del_tu(t_0,x) = u_1, \quad t_0\geq 2
	\end{equation}
	with $u_0,u_1$ sufficiently regular and compactly supported in $\{|x|<t_0-1\}$. Suppose that
	$$
	|u_0(x)| + |u_1(x)| + |\del u_0(x)|\leq C_I.
	$$
	Then for $(t,x)\in\Kcal = \{r<t-1\}$ and $t\geq t_0$, 
	\begin{equation}\label{eq1-17-08-2020}
	|u(t,x)|\leq CC_I t_0 s^{-1}\leq CC_It_0(t-r)^{-1/2}t^{-1/2}, \quad s=\sqrt{t^2-|x|^2}.
	\end{equation}
\end{lemma}
This is a classical result. In \cite{M-2020-strong} we showed a proof.

\paragraph*{$L^{\infty}$ estimate on wave equation based on integration along hyperbolas.}
We also need the following bounds to establish the sharp decay bounds without uniformly energy bounds. This is established in \cite{M-2020-strong}. We recall the following curves:
$$
\aligned
\gamma_{t,x}: \RR&\rightarrow \RR^{2+1}
\\
\tau&\rightarrow \big(\gamma_{t,x}^0(\tau),\gamma_{t,x}^1(\tau),\gamma_{t,x}^2(\tau)\big)
\endaligned
$$
with
$$
\gamma_{t,x}^0(\tau) = \tau,\quad \gamma_{t,x}^a(\tau) =  (x^a/r)\left(\sqrt{\tau^2+\frac{1}{4}C_{t,x}^2} - \frac{1}{2}C_{t,x}\right)
$$
where
$$
C_{t,x} = \frac{t^2-r^2}{r}.
$$
These are (time-like) hyperbolas with center at $(0,-\frac{x^a}{2r}C_{t,x})$ and hyperbolic radius $\frac{1}{2}C_{t,x}$. 
Then we recall the following estimate:
\begin{proposition}\label{prpo2 wave-sharp}
	Let $u$ be a sufficiently regular function defined in $\Hcal_{[s_0,s_1]}$, vanishes near $\del \Kcal = \{r=t-1\}$. 
	Then the following bound holds:
	\begin{equation}\label{eq1-29-05-2020}
	|s\del_t u(t,x)|\leq Cs_0\|\del_t u\|_{L^{\infty}(\Hcal_{s_0})} 
	+  C\bigg|\int_{s_0}^t W_{t,x}[u](\tau) e^{-\int_{\tau}^t P_{t,x}(\eta)d\eta} d\tau\bigg|
	\end{equation}
	where 
	$$
	W_{t,x}(\tau) := S^w[u]\Big|_{\gamma(\tau;t,x)} + \Delta^w[u]\Big|_{\gamma(\tau;t,x)}
	$$
	and
	$$
	P_{t,x}(\tau) := P\Big|_{\gamma(\tau;t,x)}. 
	$$
	with
$$
\aligned
&P(t,r) := \frac{t^2}{t^2+r^2}p(t,r) = \frac{t-r}{t^2+r^2}(1+(3r/2t)) \geq \frac{1}{4}(s/t)^2t^{-1},
\\
&S^w[u] := t^{1/2}(t-r)^{1/2}\frac{t^2\Box u}{t^2+r^2},\quad 
\Delta^w[u] := t^{1/2}(t-r)^{1/2}\frac{t^2\sum_a\delu_a\delu_a u}{t^2+r^2}.
\endaligned
$$
\end{proposition}

\subsection{Linear estimate on Klein-Gordon equation}

\paragraph*{Conical decay of Klein-Gordon component.}
As explained before, one of the important techniques we applied in this paper is "paying conical for principle" (see \cite{M-2020-strong}), hence we need the following proposition describing the conical decay of Klein-Gordon component.
\begin{proposition}\label{prop1-fast-kg}
	Let $v$ be a sufficiently regular solution to
	\begin{equation}\label{eq1 prop fast-KG}
	\Box v + c^2 v = f.
	\end{equation}
	Then 
	\begin{equation}\label{eq2 prop fast-KG}
	c^2|v|_{p,k}\leq C(s/t)^2|\del v|_{p+1,k+1} + C|f|_{p,k}.
	\end{equation}
\end{proposition}
Here remark the factor $(s/t)^2$ in right-hand-side. This bound is closely related to the Proof of Proposition \ref{prop1-14-08-2020}. A sketch of proof can be found in \cite{M-2020-strong}.

\paragraph*{$L^{\infty}-L^{\infty}$ estimate on Klein-Gordon component.}
Now, we reformulate the $L^\infty-L^\infty$ estimate on Klein-Gordon component for the sharp decay. Before the main statement, we introduce the following curves:
$$
\aligned
\varphi_{t,x}: \RR &\rightarrow \{(t',x')\in \RR^{2+1}, t'>0\}
\\
\lambda &\rightarrow (\lambda t/s,\lambda x/s)
\endaligned
$$
which are the half-lines from $(0,0)$ to $(t,x)$. They are the integral curves of $\Lcal = (s/t)^{-1}\big(\del_t + (x^a/t)\del_a\big)$.  For each $(t,x)\in \Hcal_{[s_0,s_1]}$, there exists a $(t_0,x_0)$ such that $(t_0,x_0)\in \varphi_{t,x}$ and $(t_0,x_0)\in \Hcal_{s_0}^*\cup \del \Kcal$. Here $\Hcal_{s_0}^* = \Hcal_{s_0}\cap\Kcal$ is the part of $\Hcal_{s_0}$ in the cone $\Kcal$. Then we state the main result:
\begin{proposition}\label{lem1'-01-09-2020}
	Suppose that $v$ is a $C^2$ solution to the following Klein-Gordon equation:
	\begin{equation}\label{eq'14-01-09-2020}
	\Box v + c^2(1-\omega)v = f
	\end{equation}
	in $\Hcal_{[s_0,s_1]}$, vanishes near $\del K$ with $s_0\geq 2$. Suppose that $\omega$ and $f$ $C^1$ functions defined in $\Hcal_{[s_0,s_1]}$, vanish near $\del\Kcal$ with $|w_0|\leq 1/2$. Then for $(t,x)\in \Hcal_{[s_0,s_1]}$ with $0\leq r/t\leq 3/5$
	\begin{equation}
	\aligned
	s|v|(t,x) + &s|((s/t)\del_t + (x^a/s)\delu_a)v|(t,x)
	\\
	\leq&
	Cs_0\sup_{\Hcal_{s_0}}\{|v| + |\del v|\} 
	+ C\int_{s_0}^s\lambda\big(|f| + \lambda^{-2}|v|_{2,2}\big)\big|_{\varphi_{t,x}(\lambda)} d\lambda
	\\
	&+ C\int_{s_0}^s\big(\lambda|\del v| + \lambda|v| + |v|_{1,1}\big)|\del \omega|\big|_{\varphi_{t,x}(\lambda)}d\lambda
	\endaligned
	\end{equation}
	and for $(t,x)\in \Hcal_{[s_0,s_1]}$ with $3/5\leq r/t<1$,
	\begin{equation}
	\aligned
	s|v|(t,x)& + s|((s/t)\del_t + (x^a/s)\delu_a)v|(t,x) 
	\\
	\leq&
	C\int_{\lambda_0}^s\lambda\big(|f| + \lambda^{-2}|v|_{2,2}\big)\big|_{\varphi_{t,x}(\lambda)} d\lambda
	\\
	&+ C(s/t)^{-1}\int_{\lambda_0}^s\big((s/t)\lambda|\del v| + \lambda|v| + |v|_{1,1}\big)\big((s/t)^2|\del \omega| + |\dels \omega|)\big|_{\varphi_{t,x}(\lambda)}d\lambda.
	\endaligned
	\end{equation}
with $\lambda_0 = \sqrt{\frac{t+r}{t-r}}\geq \sqrt{2}(t/s)$.
\end{proposition}

\begin{proof}[Sketch of proof]
	Recall the following decomposition
	$$
	\Box v = s^{-1}\big((s/t)\del_s + (x^a/s)\delu_a\big)^2(sv) - \frac{x^ax^b}{s^2}\delu_a\delu_b v - \sum_{a}\delu_a\delu_a v .
	$$ 
	Then \eqref{eq'14-01-09-2020} can be written into the following form:
	\begin{equation}\label{eq'16-01-09-2020}
	\Lcal^2(sv) + c^2(1-\omega)sv = sf + s\big(s^{-2}x^ax^b\delu_a\delu_b + \sum_a\delu_a\delu_a \big)v
	\end{equation}
	where $\Lcal = (s/t)\del_t + (x^a/s)\delu_a = (s/t)^{-1}\big(\del_t  + (x^a/t)\del_a\big)$. This can be regarded as an ODE of $sv$ along the integral curve of $\Lcal$, which are segments. Let $\varphi_{t,x}(\cdot)$ be one of its integral curve such that $\varphi_{t,x}(s) = (t,x)$ with $s = \sqrt{t^2-r^2}$. Then
	$$
	\varphi^0_{t,x}(\lambda) = (t/s)\lambda,\quad \varphi^a_{t,x}(\lambda) = (x^a/s)\lambda.
	$$
	Let $u$ be a sufficiently regular function defined in $\Hcal_{[s_0,s_1]}$, and 
	$$
	u_{t,x}(\lambda):= u|_{\varphi_{t,x}(\lambda)} = u\big((t/s)\lambda, (x^a/s)\lambda\big),
	$$
	then 
	$$
	u'_{t,x}(\lambda) = \frac{d}{d\lambda}u_{t,x}(\lambda) = \Lcal u \big((t/s)\lambda, (x^a/s)\lambda\big) = (\Lcal u)|_{\varphi_{t,x}(\lambda)}.
	$$
	With these observation, \eqref{eq'16-01-09-2020} is written as 
	\begin{equation}\label{eq'17-01-09-2020}
	V_{t,x}''(\lambda) + c^2(1-\omega)V_{t,x}(\lambda) = \lambda \big(f + s^{-2}x^ax^b\delu_a\delu_bv + \sum_a\delu_a\delu_av \big)\big|_{\varphi_{t,x}(\lambda)}
	\end{equation}
	where $V_{t,x}(\lambda) : = (sv)|_{\varphi_{t,x}}(\lambda)$. Here we also remarked that for $(t',x')\in \varphi_{t,x}(\lambda)$ and $s' = \sqrt{|t'|^2 - |x'|^2}$, 
	$$
	s'|_{\varphi_{t,x}} = \lambda'
	$$
	with $t' = \lambda' t/s, x' = \lambda' x/s$.
	
	Then we make an observation on the integral curves $\varphi_{t,x}$. They are half-lines form $(0,0)$ to $(t,x)\in  \Hcal_{[s_0,s_1]}$. Recall that $(t_0,x_0)$ is the point where $\varphi_{t,x}$ enters $\Hcal_{[s_0,s_1]}$. Direct calculation shows that:
	\\
	- when $0\leq r/t\leq 3/5$, $(t_0,x_0)\in \Hcal_{s_0}$. Let $(t_0,x_0) = \varphi_{t,x}(\lambda_0)$, then $\lambda_0=2$,
	\\
	- when $3/5\leq r/t <1$, $(t_0,x_0)\in\del\Kcal = \{t=r+1\}$ and $\lambda_0 = \sqrt{\frac{t+r}{t-r}}\geq \sqrt{2}t/s$.
	
	Now for a fixed $(t,x)\in \Hcal_s^*\subset \Hcal_{[s_0,s_1]}$ we integrate \eqref{eq'17-01-09-2020}. Remark that that $|\omega|\leq 1/2$, the eigenvalues of the characteristic polynomial are purely imaginary and the eigenvectors are uniformly bounded. So by basic ODE theory we arrive at the following bound:
	\begin{equation}\label{eq9-08-10-2020}
	\aligned
	|V_{t,x}'(s)| + |V_{t,x}(s)|\leq& |V_{t,x}'(\lambda_0)| + |V_{t,x}(\lambda_0)| 
	\\
	&+ C\int_{\lambda_0}^s\lambda \big(|f| + |s^{-2}x^ax^b\delu_a\delu_bv| + \sum_a|\delu_a\delu_av|\big)\big|_{\varphi_{t,x}(\lambda)}d\lambda 
	\\
	& + C\int_{\lambda_0}^s(|V_{t,x}(\lambda)| 
	+ |V'_{t,x}(\lambda)|) |\Lcal \omega|\big|_{\varphi_{t,x}(\lambda)}d\lambda.
	\endaligned
	\end{equation}
	Then  remark that
	\\
	1. $(s/t)$ and $(r/t)$ are constant along $\varphi_{t,x}$ and when $0\leq r/t\leq 3/5$, $4/5\leq s/t\leq 1$, i.e., we can omit all factors $(s/t)$ (regarded as $1$).
	\\
	2. When $0\leq r/t\leq 3/5$, $|V_{t,x}(\lambda_0)| + |V_{t,x}'(\lambda_0)|\leq Cs_0\sup_{\Hcal_{s_0}}\{|v| + |\del v|\}$,
	\\
	3. When $3/5\leq r/t<1$, $|V_{t,x}(\lambda_0)| + |V_{t,x}'(\lambda_0)|=0$.
	
Furthermore,
$$
\aligned
|V_{t,x}(\lambda)| + |V_{t,x}'(\lambda)|\leq& (\lambda + 1)|v(\lambda t/s,\lambda x/s)| +  \lambda(s/t)|\del_t v(\lambda t/s,\lambda x/s)| 
\\
&+ C(s/t)^{-1}\lambda |\dels v(\lambda t/s,\lambda x/s)|
\\
\leq& C(\lambda + 1)|v|_{\varphi_{t,x}(\lambda)} + C\lambda(s/t)|\del v|_{\varphi_{t,x}(\lambda)} + C|L v|_{\varphi_{t,x}(\lambda)},
\endaligned
$$
$$
|\delu_a\delu_b v|\leq Ct^{-2}|v|_{2,2},
$$
$$
|\Lcal \omega| = |(s/t)\del_t \omega + (x^a/s)\delu_a\omega| 
\leq C(s/t)|\del \omega| + C(s/t)^{-1}|\dels \omega|.
$$
Then \eqref{eq9-08-10-2020} is written as 
$$
\aligned
|& sv(t,x)| +|v(t,x) + s\big((s/t)\del_tv + (x^a/s)\delu_a v\big)| 
\\
\leq&  C\int_{\lambda_0}^s\lambda(|f| + \lambda^{-2}|v|_{2,2})\big|_{\varphi_{t,x}(\lambda)}d\lambda
\\
&+ C\int_{\lambda_0}^s\lambda \big(|v| + (s/t)|\del v| + \lambda^{-1}|v|_{1,1}\big)\big((s/t)|\del \omega| + (s/t)^{-1}|\dels\omega|_{1,1}\big)\big|_{\varphi_{t,x}(\lambda)}d\lambda 
\\
&+
\left\{
\begin{aligned}
&Cs_0\sup_{\Hcal_{s_0}} \{|v| + |\del v|\},\quad &&0\leq r/t\leq 3/5,
\\
&0,\quad &&3/5\leq r/t<1.
\end{aligned}
\right.
\endaligned
$$
which concludes the desired result.
\end{proof}

\section{Reformulation of the systems}\label{sec-reformulation}
This section together with the following one are devoted to the model systems \eqref{eq1-main}. In this section we construct the auxiliary system \eqref{eq1-05-09-2020}.
\subsection{Construction of the auxiliary systems}
We get stated with \eqref{eq1a-main}. For any $C^3$ solution $(u,v)$ to \eqref{eq1a-main}, suppose that $w$ is a solution to the following wave equation:
$$
\Box w = v^2.
$$ 
Then 
$$
\Box \big(A^{\alpha}\del_{\alpha}w\big) = A^{\alpha}\del_{\alpha}(v^2),
$$
which shows that $A^{\alpha}\del_{\alpha}w$ and $u$ satisfy the same wave equation. Based on this observation, we make the following reformulation. Let $(u,v)$ be a $C^3$ solution to the Cauchy problem of \eqref{eq1-main} with the initial data
\begin{equation}\label{eq3-03-09-2020}
u(2,x) = u_0(x), \quad \del_t u(2,x) = u_1(x),\quad v(2,x) = v_0(x),\quad \del_tv(2,x) = v_1.
\end{equation}
Then for the following auxiliary Cauchy problem
\begin{equation}\label{eq4-03-09-2020}
\left\{
\aligned
&\Box w = v^2,
\\
&\Box w_0 = 0,
\\
&\Box \tilde{v} + c^2\tilde{v} = B^{\alpha}\tilde{v}\del_{\alpha}(w_0 + A^{\beta}\del_{\beta}w).
\endaligned
\right.
\end{equation} 
with
\begin{equation}\label{eq5-03-09-2020}
\aligned
&w(2,x) = \del_tw(2,x) = 0,\quad \tilde{v}(2,x) = v_0(x),\quad \del_t \tilde{v}(2,x) = v_1(x)
\\
&w_0(2,x) = u_0(x), \quad \del_tw_0(2,x) = u_1(x) - A^0v_0^2(x),
\endaligned
\end{equation}
we can establish the following result:
\begin{lemma}\label{lem1-05-10-2020}
	Let $(u,v)$ be the $C^3$ solution to the Cauchy problem associate to \eqref{eq1-main} with initial data \eqref{eq3-03-09-2020} and $(w,w_0,\tilde{v})$ be the $C^2$ solution to \eqref{eq3-03-09-2020} with initial data \eqref{eq5-03-09-2020}. Then
	\begin{equation}\label{eq1-08-09-2020}
	u = w_0 + A^{\alpha}\del_{\alpha}w,\quad v = \tilde{v}
	\end{equation}
	when both solutions exist. Furthermore, when $(w,w_0,\tilde{v})$ exists, $(u,v)$ defined through \eqref{eq1-08-09-2020} is the solution to \eqref{eq1a-main} with \eqref{eq3-03-09-2020}.
\end{lemma}
\begin{proof}
	This is an argument based on the uniqueness of \eqref{eq1-main}. In fact we calculate
	\begin{equation}\label{eq8-03-09-2020}
	\Box \big(w_0 + A^{\alpha}\del_{\alpha}u\big) = \Box w_0 + A^{\alpha}\del_{\alpha}\Box w = A^{\alpha}\del_{\alpha}(v^2)
	\end{equation}
	and on the initial slice,
	\begin{equation}\label{eq7-03-09-2020}
	w_0(2,x) + A^{\alpha}\del_{\alpha}w(2,x) = u_0(x).
	\end{equation}
	On the other hand,
	$$
	\del_t(w_0 + A^{\alpha}\del_{\alpha}w) = \del_tw_0 + A^0\del_t\del_tw + A^a\del_t\del_aw.
	$$
	On $\{t=2\}$, recall that $\del_tw = \del_aw = 0$,
	$$
	\del_t(w_0 + A^{\alpha}\del_{\alpha}w)(2,x) = \del_tw_0(2,x) + A^0\del_t\del_t w(2,x) = u_1(x) - A^0v_0^2(x) + A^0\del_t\del_t w(2,x).
	$$
	Furthermore, remark that:
	$$
	\Box w = v^2\ \Rightarrow\ \del_t\del_t w(2,x) = v^2(2,x) + \sum_a\del_a\del_au(2,x) = v_0^2(x).
	$$
	Substitute this into the last expression, we obtain:
	\begin{equation}\label{eq6-03-09-2020}
	\del_t(w_0 + A^{\alpha}\del_{\alpha}w)(2,x) = u_1(x).
	\end{equation}
	
	Let $\tilde{u} = w_0+A^{\alpha}\del_{\alpha}w$.
	Consider \eqref{eq8-03-09-2020}, \eqref{eq7-03-09-2020} and \eqref{eq6-03-09-2020}, $(\tilde{u}, \tilde{v})$ satisfies the following Cauchy Problem
	$$
	\left\{
	\aligned
	&\Box \tilde{u} = A^{\alpha}\del_{\alpha}(v^2)
	\\
	&\Box \tilde{v}  + c^2\tilde{v} = B^{\alpha}\tilde{v}\del_{\alpha}\tilde{u}
	\endaligned
	\right.
	$$
	with initial data 
	$$
	\tilde{u}(2,x) = u_0(x), \quad \del_t \tilde{u}(2,x) = u_1(x),\quad \tilde{v}(2,x) = v_0(x),\quad \del_t\tilde{v}(2,x) = v_1.
	$$
	Then by uniqueness theory on \eqref{eq1a-main}, the desired result is obtained.
\end{proof}

Similar to \eqref{eq1a-main}, \eqref{eq1b-main} can be reformulated as following. We consider 
\begin{equation}\label{eq1-05-09-2020}
\left\{
\aligned
&\Box w = v^2,
\\
&\Box w_0 = 0,
\\
&\Box \tilde{v} + c^2\tilde{v} = B\tilde{v}(w_0 + A^{\alpha\beta}\del_{\alpha}\del_{\beta}w)
\endaligned
\right.
\end{equation}
with initial data constructed as following:
\begin{equation}\label{eq2-05-09-2020}
\aligned
&w(2,x) = \del_tw(2,x) = 0,\quad \tilde{v}(2,x) = v_0(x),\quad \del_t\tilde{v}(2,x) = v_1(x)
\\
&w_0(2,x) = u_0(x),\quad \del_tw_0(2,x) = u_1(x) - 2A^{00}v_0(x)v_1(x).
\endaligned
\end{equation}
Then similar to the previous result, one has
\begin{lemma}\label{lem1-06-10-2020}
	Let $(u,v)$ be a $C^4$ solution to \eqref{eq1b-main} with the following initial data
	\begin{equation}\label{eq3-08-09-2020}
	u(2,x) = u_0(x),\quad \del_t u(2,x) = u_1(x),\quad v(2,x) = v_0(x),\quad \del_t v(2,x) = v_1(x).
	\end{equation}
	Suppose that $(w,w_0,\tilde{v})$ is the $C^2$ solution to the Cauchy problem of \eqref{eq1-05-09-2020} with initial data \eqref{eq2-05-09-2020}. Then
	\begin{equation}\label{eq2-08-09-2020}
	u = w_0 + A^{\alpha\beta}\del_{\alpha}\del_{\beta}w,\quad v = \tilde{v}
	\end{equation}
	when both $(u,v)$ and $(w,w_0,\tilde{v})$ exist. Furthermore, when $(w,w_0,\tilde{v})$ exists, $(u,v)$ defined via \eqref{eq2-08-09-2020} is the solution to \eqref{eq1b-main} with \eqref{eq3-08-09-2020}.
\end{lemma}
\begin{proof}
	The proof is quite similar. We only need to worry about the local uniqueness of the system \eqref{eq1b-main}. This is can be regarded as an application of Theorem 2.2 in Section 1.2 of \cite{Sogge-2008-book}.
\end{proof}

In order to treat \eqref{eq4-03-09-2020} and \eqref{eq1-05-09-2020}  simultaneously, we consider a more general system
\begin{equation}\label{eq-main}
\left\{
\aligned
&\Box w  = v^2,
\\
&\Box w_0 = 0,
\\
&\Box v + c^2v = B^{\alpha}v\del_{\alpha}w_0 + Kvw_0 + vA^{\alpha\beta}\del_{\alpha}\del_\beta w.
\endaligned
\right.
\end{equation} 

\subsection{Statement of the main result on auxiliary system}\label{subsec-auxi-structure}
As explained in Introduction, we will firstly establish global stability results on \eqref{eq-main}.
\begin{theorem}\label{thm-auxiliary}
Consider the Cauchy problem associated to \eqref{eq-main} with initial data posed on $\{t=2\}$ and compactly supported in $\{|x|<1\}$:
\begin{align*}
&w(2,x) = \del_tw(2,x) = 0,\quad v(2,x) = v_0(x),\quad \del_tv(2,x) = v_1(x),
\\
&w_0(2,x) = u_0(x),\quad \del_tw_0(2,x) = u_1(x).
\end{align*}
Then there exists a integer $N\geq 9$ and a positive constant $\vep_0>0$ determined by the system, such that for all $0\leq \vep\leq \vep_0$, if
\begin{equation}
\|u_0\|_{H^{N+1}} + \|v_0\|_{H^{N+1}} + \|u_1\|_{L^2(H^N)} + \|v_1\|_{H^N}\leq \vep,
\end{equation}
then the local-in-time solution of \eqref{eq-main} associated with such initial data extends to time infinity.
\end{theorem}

Based on the above result together with Lemma \ref{lem1-05-10-2020}, Lemma \ref{lem1-06-10-2020}, we conclude Theorem \ref{thm-main}.

\begin{remark}
For the Cauchy problem associated to \eqref{eq-main}, one can also consider a initial data with non-zero $w(2,x)$, $\del_tw(2,x)$ and the global stability result still holds.  
\end{remark}

\subsection{Structure of the auxiliary system}
\eqref{eq-main} is still a strong coupled W-KG system. However, it enjoys a special structure called the Hessian structure. That is, omit for a moment the linear component $w_0$, the wave component $w$ is only coupled in Hessian form in right-hand-side of the system and especially, the gradient $\del w$ does not appear. As explained in Introduction, the better decay and energy bounds of Hessian form permits us to establish integrabel $L^2$ bound on $v\del\del w$. 

If we omit $w_0$, then this kind of system has already been handled in  \cite{Stingo-2018}, \cite{Dong-2020-2} and \cite{M-2020-strong}. Following the perspective of \cite{M-2020-strong}, we aquire that this system is {\sl subcritical} in the sense of principle decay. 

However, the presence of $w_0$ brings supplementary terms $vw_0,v\del w_0$  which are not completely trivial. 
Given that the decay of $|w_0|$ and $|\del w_0|$ are $s^{-1}$ which seems impossible to be improved, $vw_0$ and $v\del w_0$ will lead at least a logarithmic loss on the energy bound of Klein-Gordon component. This prevents one from expecting uniform energy bounds on Klein-Gordon component for lower (even for zero) order. Without this important uniform bound, one can no longer obtain sharp decay $v\sim t^{-1}$ via Klainerman-Sobolev inequality, which is crucial in the bootstrap argument. To overpass this difficulty, we rely on Proposition \ref{lem1'-01-09-2020}. This $L^{\infty}-L^{\infty}$ estimate is originally introduced in \cite{Kl2} and applied in many other context, see for example \cite{Dfx}, \cite{KS-2011}, \cite{M1} etc. Here we present a version with non-constant-coefficient Klein-Gordon potential. This permits us to establish the following decay
$$
|\del v|\simeq (s/t)^2s^{-1}
$$
without uniform energy bounds. Similarly, this lack of uniform energy bound on Klein-Gordon component also brings inconvenience when we try to obtain sharp decay on $\del\del w$ because of the term $v^2$ coupled in the equation of $w$.  This relies on a $L^{\infty}-L^{\infty}$ estimate on wave equation based on integration along hyperbolas which is Proposition \ref{prpo2 wave-sharp} established in \cite{M-2020-strong}.

%
%


\section{Proof of Theorem \ref{thm-auxiliary}}\label{sec-bootstrap}
\subsection{Energy and decay bounds on $w_0$}
Remark that $w_0$ is a solution to a free-linear wave equation with sufficiently regular and compactly supported initial data. Then it has conserved standard and conformal energy:
\begin{equation}
\Ecal_0^N(s,w_0)^{1/2} + \Ecal_2^N(s,w_0)^{1/2}\leq C_0\vep
\end{equation}
where $C_0$ is a constant determined by $N$. Standard energy bounds lead to the following decay:
\begin{equation}\label{eq6-08-10-2020}
|\del w_0|_{N-2}\leq CC_0\vep s^{-1},\quad |\dels w_0|_{N-2}\leq CC_0\vep t^{-1}.
\end{equation}
Remark that in this case, recalling \eqref{eq1-09-10-2020},
\begin{equation}\label{eq8-08-10-2020}
\Fcal_2^N(2;s,w_0)\leq CC_1\vep\ln(s)\leq CC_1\vep s^{\delta}.
\end{equation}
Then by \eqref{eq2-10-06-2020},
\begin{equation}\label{eq8-08-09-2020}
|\del w_0|_{N-2}\leq CC_0\vep (s/t)^{-1}s^{-2}\ln(s)\leq CC_1\vep (s/t)^{-1}s^{-2+\delta}.
\end{equation}
Furthermore, by Lemma \ref{lem1-13-06-2020}
\begin{equation}\label{eq3-01-09-2020}
|w_0|_{N-2}\leq CC_1\vep s^{-1}
\end{equation}
which leads to
\begin{equation}\label{eq9-08-09-2020}
|\dels w_0|_{N-3}\leq CC_0\vep (s/t)s^{-2}.
\end{equation}

\subsection{Bootstrap assumption and direct bounds}
Remark that the initial data are posed on $\{t=1\}$ and supported in $\{|x|<1\}$. The property of finite speed propagation says that the local solution is supported in $\Kcal = \{r<t-1\}$. Furthermore, taking $\vep$ sufficiently small such that (thanks to the local theory on wave system) the local solution extends beyond $t=5/2$ and remark that $\Hcal_2\cap\Kcal\subset \{2\leq t\leq 5/2\}$, one can take the restriction of the local solution on $\Hcal_2$ as the initial data on $\Hcal_2$. Again, due to the local theory, the energy on $\Hcal_2$ is bounded by the initial energy on $\{t=2\}$.  So for sufficiently small $\vep$ (determined by the system and $N$), there is a constant $C_0$ (also determined by the system and $N$) such that
$$
\max\Big\{\sum_{\alpha=0}^3\Ecal_0^N(2,\del_{\alpha} w)^{1/2}, \Ecal_0^N(2,w)^{1/2}, \Ecal_{0,1}^N(2,v)^{1/2}\Big\} = C_0\vep .
$$

Then make the following bootstrap assumption on a hyperbolic time interval $[2,s_1]$:
\begin{equation}\label{eq8-15-08-2020}
\max\Big\{\sum_{\alpha}\Ecal_0^N(s,\del_{\alpha} w)^{1/2}, \Ecal_0^N(s,w)^{1/2}, \Ecal_{0,1}^N(s,v)^{1/2}\Big\}\leq C_1\vep s^{\delta}
\end{equation}
with $C_1> C_0, \delta\leq 1/20$.

\begin{remark}
	The restriction on $N$ can be improved. However here we simply take $N\geq 9$ because when considering $|AB|_p, p\leq N$, we want 
	$$
	|AB|_p\leq C|A|_p|B|_{N-5} + C|A|_{N-5}|B|_p.
	$$
\end{remark}

By Klainerman-Sobolev type inequality, 
\begin{equation}\label{eq1-08-10-2020}
s|\del\del w|_{N-2} + t|\del\dels w|_{N-2}\leq CC_1\vep s^{\delta},
\end{equation}
\begin{equation}\label{eq10-15-08-2020}
s|\del w|_{N-2} + t|\dels w|_{N-2} \leq CC_1\vep s^{\delta},
\end{equation}
\begin{equation}\label{eq11-15-08-2020}
s|\del v|_{N-2} + t|\dels v|_{N-2} + t|v|_{N-2}\leq CC_1\vep s^{\delta}.
\end{equation}

\begin{remark}
	During the analysis, $C$ denotes a constant determined by $N,\delta$ and system. 
\end{remark}

\subsection{Bounds on Hessian form of $w$.}
By Proposition \ref{prop1-14-08-2020}, one can establish the following bounds on Hessian form:
\begin{equation}\label{eq12-15-08-2020}
\|(s/t)^2s|\del\del w|_{N-1}\|_{L^2(\Hcal_s)}\leq  CC_1\vep s^{2\delta},
\end{equation}
\begin{equation}\label{eq13-15-08-2020}
(s/t)^2|\del\del w|_{N-3}\leq CC_1\vep (s/t)s^{-2+2\delta}.
\end{equation}
Here we remark that the Hessian form enjoy better principle decay ($-2+2\delta$ order) than the gradient ($-1+\delta$ order).

These are based on the following bounds on $|\Box w| = |v^2|$:
\begin{equation}\label{eq2-08-10-2020}
|\Box w|_{N-2}\leq C(C_1\vep)^2(s/t)^2s^{-2+2\delta},
\end{equation}
\begin{equation}\label{eq3-08-10-2020}
\|(s/t)^{-1}|\Box w|_N\|_{L^2(\Hcal_s)}\leq CC_1\vep s^{-1+2\delta}.
\end{equation}
The first is direct via \eqref{eq11-15-08-2020}. For the second, remark that
$$
\aligned
\|(s/t)^{-1}|v^2|_N\|_{L^2(\Hcal_s)}\leq& CC_1\vep\|(s/t)^{-1}|v|_{N-2}|v|_N\|_{L^2(\Hcal_s)}
\\
\leq& CC_1\vep s^{-1+\delta} \||v|_N\|_{L^2(\Hcal_s)}\leq C(C_1\vep)^2s^{-1+2\delta}. 
\endaligned
$$
Then apply Proposition \ref{prop1-14-08-2020} together with \eqref{eq2-08-10-2020} and \eqref{eq3-08-10-2020}, \eqref{eq12-15-08-2020} and \eqref{eq13-15-08-2020} are proved.

\subsection{Conical decay of Klein-Gordon component}\label{subsec-model-conical}
In this subsection we establish the following two bounds:
\begin{equation}\label{eq16-15-08-2020}
|v|_{N-3}\leq CC_1\vep(s/t)^2s^{-1+\delta}.
\end{equation}
\begin{equation}\label{eq17-15-08-2020}
\|(s/t)^{-1}|v|_{N-1}\|_{L^2(\Hcal_s)}\leq CC_1\vep s^{\delta}.
\end{equation}

These are done by applying Proposition \ref{prop1-fast-kg}. We first prove that
\begin{equation}\label{eq4-08-10-2020}
|\Box v + c^2 v|_{N-3}\leq CC_1\vep |v|_{N-3}.
\end{equation}
This is by checking each term in right-hand-side of the Klein-Gordon equation of \eqref{eq-main}. In fact by \eqref{eq3-01-09-2020} and \eqref{eq6-08-10-2020},
$$
|vw_0|_{N-3} + |v\del w_0|_{N-3}\leq C\big(|w_0|_{N-3} + |\del w_0|_{N-3}\big)|v|_{N-3}\leq CC_1\vep|v|_{N-3}.
$$
Finally, by \eqref{eq1-08-10-2020}
$$
|v\del\del w|_{N-3}\leq C|\del\del w|_{N-3}|v|_{N-3}\leq CC_1\vep |v|_{N-3}.
$$
So we conclude by \eqref{eq4-08-10-2020}. Then substitute \eqref{eq4-08-10-2020} into \eqref{eq2 prop fast-KG},
$$
c^2|v|_{N-3}\leq C(s/t)^2|\del v|_{N-2} + C|\Box v + c^2 v|_{N-3}\leq CC_1\vep(s/t)^2s^{-1+\delta} + CC_1\vep|v|_{N-3}.
$$
Taking $C_1\vep$ sufficiently small such that
\begin{equation}\label{eq7-08-10-2020}
|CC_1\vep|\leq \frac{c^2}{2},
\end{equation}
we obtain \eqref{eq16-15-08-2020}.

Then we turn to $L^2$ bound. We establish the following bound on source term:
\begin{equation}\label{eq5-08-10-2020}
\|(s/t)^{-1}|\Box v + c^2 v|_{N-1}\|_{L^2(\Hcal_s)}\leq C(C_1\vep)^2s^{\delta} + CC_1\vep\|(s/t)^{-1}|v|_{N-1}\|_{L^2(\Hcal_s)} .
\end{equation}
This is also by checking each term. 
$$
\begin{aligned}
&\|(s/t)^{-1}|vw_0|_{N-1}\|_{L^2(\Hcal_s)}
\\
\leq& C\|(s/t)^{-1}|w_0|_{N-2}|v|_{N-1}\|_{L^2(\Hcal_s)} 
\\
&+ C\|(s/t)^{-1}|v|_{N-3}|w_0|_{N-1}\|_{L^2(\Hcal_s)}
\\
\leq& CC_1\vep\|(s/t)^{-1}s^{-1}|v|_{N-1}\|_{L^2(\Hcal_s)}
 + CC_1\vep s^{-1+\delta}\|(s/t)|w_0|_{N-1}\|_{L^2(\Hcal_s)}
\\
\leq& CC_1\vep\|(s/t)^{-1}|v|_{N-1}\|_{L^2(\Hcal_s)}
 + C(C_1\vep)^2 s^{-1+\delta}\ln (s)
\end{aligned}
$$
where \eqref{eq3-01-09-2020}, \eqref{eq16-15-08-2020} and \eqref{eq2-10-06-2020} (combined with \eqref{eq8-08-10-2020}) are applied. Similarly,
$$
\aligned
&\|(s/t)^{-1}|v\del w_0|_{N-1}\|_{L^2(\Hcal_s)}
\\
\leq& C\|(s/t)^{-1}|v|_{N-1}|\del w_0|_{N-2}\|_{L^2(\Hcal_s)} + \|(s/t)^{-1}|v|_{N-3}|\del w_0|_{N-1}\|_{L^2(\Hcal_s)}
\\
\leq& CC_1\vep\|(s/t)^{-1}|v|_{N-1}\|_{L^2(\Hcal_s)} + CC_1\vep s^{-1+\delta}\|(s/t)|\del w_0|_{N-1}\|_{L^2(\Hcal_s)}
\\
\leq& CC_1\vep\|(s/t)^{-1}|v|_{N-1}\|_{L^2(\Hcal_s)} +  C(C_1\vep)^2s^{-1+\delta}.
\endaligned
$$
$$
\aligned
&\|(s/t)^{-1}|v\del\del w|_{N-1}\|_{L^2(\Hcal_s)}
\\
\leq& C\|(s/t)^{-1}|v|_{N-1}|\del\del w|_{N-2}\|_{L^2(\Hcal_s)}
\\
&+ C\|(s/t)^{-1}|v|_{N-3}|\del\del w|_{N-1}\|_{L^2(\Hcal_s)}
\\
\leq& CC_1\vep \|(s/t)^{-1}|v|_{N-1}\|_{L^2(\Hcal_s)} + CC_1\vep s^{-1+\delta}\|(s/t)|\del\del w|_{N-1}\|_{L^2(\Hcal_s)}
\\
\leq& CC_1\vep \|(s/t)^{-1}|v|_{N-1}\|_{L^2(\Hcal_s)} + C(C_1\vep)^2 s^{-1+2\delta}.
\endaligned
$$
Then we conclude by \eqref{eq5-08-10-2020}.

Now taking $C_1\vep$ sufficiently small, simlar to \eqref{eq16-15-08-2020}, \eqref{eq17-15-08-2020} is established.

\subsection{Improved energy bounds for lower order: Klein-Gordon component}
\label{subsec-model-KG-lower}
This subsection is dedicated to 
\begin{equation}\label{eq6-01-09-2020}
\Ecal_{0,c}^{N-2}(s,v)^{1/2}\leq \big(C_0\vep  + C(C_1\vep)^2\big)(s/2)^{CC_1\vep}.
\end{equation}
Let us firstly establish the following bounds on source terms.
\begin{equation}\label{eq4-01-09-2020}
\||v\del\del w|_{N-1}\|_{L^2(\Hcal_s)}\leq C(C_1\vep)^2s^{-2+3\delta},
\end{equation}
\begin{equation}\label{eq5-01-09-2020}
\||vw_0|_{N-2}\|_{L^2(\Hcal_s)} + \||v\del w_0|_{N-2}\|_{L^2(\Hcal_s)}\leq CC_1\vep s^{-1}\Ecal_{0,c}^{N-2}(s,v)^{1/2}.
\end{equation}
The first is due to \eqref{eq12-15-08-2020}, \eqref{eq13-15-08-2020}, \eqref{eq16-15-08-2020} and \eqref{eq17-15-08-2020} :
$$
\aligned
\||v\del\del w|_{N-1}\|_{L^2(\Hcal_s)}\leq& C\||v|_{N-3}|\del\del w|_{N-1}\|_{L^2(\Hcal_s)} + C\||v|_{N-1}|\del\del w|_{N-3}\|_{L^2(\Hcal_s)}
\\
\leq& CC_1\vep s^{-2+\delta}\|(s/t)^2s|\del\del w|_{N-1}\|_{L^2(\Hcal_s)} 
\\
&+ CC_1\vep s^{-2+2\delta}\|(s/t)^{-1}|v|_{N-1}\|_{L^2(\Hcal_s)}
\\
\leq& C(C_1\vep)^2s^{-2+3\delta}.
\endaligned
$$
\eqref{eq5-01-09-2020} is directly by \eqref{eq3-01-09-2020} and \eqref{eq6-08-10-2020}. Then we conclude that, by energy estimate Proposition \ref{prop 1 energy} 
$$
\Ecal_{0,c}^{N-2}(s,v)^{1/2}\leq \Ecal_{0,c}^{N-2}(2,v)^{1/2} + C(C_1\vep)^2 
+ CC_1\vep\int_2^s\tau^{-1}\Ecal_{0,c}^{N-2}(\tau,v)^{1/2}d\tau.
$$ 
Then by Gronwall's inequality, \eqref{eq6-01-09-2020} is concluded.

A direct result of \eqref{eq6-01-09-2020} is the following bounds (thanks to Klainerman-Sobolev inequality and the fact that $C_0\leq C_1$, $C_1\vep \leq 1$) 
\begin{equation}\label{eq10-01-09-2020}
s|\del v|_{N-4} + t|v|_{N-4}\leq CC_1\vep s^{CC_1\vep}.
\end{equation}

\subsection{Sharp decay bounds.}\label{section sharp decay bounds}
This subsection is dedicated to the following sharp decay bounds.
\begin{equation}\label{eq12-01-09-2020}
|v|_{N-5}\leq CC_1\vep (s/t)^2s^{-1+CC_1\vep},
\end{equation}
\begin{equation}\label{eq18-01-09-2020}
|\del \del w|_{N-5}\leq CC_1\vep s^{-1 + CC_1\vep},
\end{equation}
\begin{equation}\label{eq11-01-09-2020}
(s/t)|\del v| + |v| \leq CC_1\vep (s/t)^2s^{-1},
\end{equation}
\begin{equation}\label{eq13-01-09-2020}
|\del\del w|\leq CC_1\vep s^{-1}.
\end{equation}

\paragraph*{Proof of \eqref{eq12-01-09-2020}.} This is the most easy one. It is based on \eqref{eq10-01-09-2020} and parallel to \eqref{eq16-15-08-2020}. In fact we recall \eqref{eq2 prop fast-KG} and \eqref{eq4-08-10-2020},
$$
\aligned
c^2|v|_{N-5}\leq& CC_1\vep (s/t)^2|\del v|_{N-4} + C|\Box v + c^2v|_{N-5}
\\
\leq& CC_1\vep (s/t)^2|\del v|_{N-4} + CC_1\vep |v|_{N-5}.
\endaligned
$$
Then by \eqref{eq10-01-09-2020} and \eqref{eq7-08-10-2020}, i.e. with $C_1\vep$ sufficiently small, \eqref{eq12-01-09-2020} is established.

\paragraph*{Proof of \eqref{eq18-01-09-2020}.} This is a direct consequence of \eqref{eq10-01-09-2020}. We apply Proposition \ref{prpo2 wave-sharp} applied on the equation satisfied by $\del_{\alpha}u$:
\begin{equation}\label{eq6-08-09-2020}
\Box \del_{\alpha}\del^IL^J w = \del_{\alpha}\del^IL^J(v^2).
\end{equation}
Following the notation of Proposition \ref{prpo2 wave-sharp}, one has
$$
\Delta^w[\del_{\alpha}\del^IL^J w] = t^{1/2}(t-r)^{1/2}\frac{t^2}{t^2+r^2}\sum_a\delu_a\delu_a\del_{\alpha}\del^IL^J w
$$
and thus for $|I|+|J|\leq N-4$ \footnote{Here we have applied $|\dels\dels u|_{p,k}\leq Ct^{-2}|u|_{p+2,k+2}$. This can be observed by homogeneity. A proof can be found in \cite{M-2020-strong}.}, 
$$
|\Delta^w[\del_{\alpha}\del^IL^J w]|\leq Cs|\dels\dels\del_{\alpha} w|_{N-4}\leq C(s/t)^2s^{-1}|\del w|_{N-2}
$$
which leads to, thanks to \eqref{eq10-15-08-2020}
\begin{equation}\label{eq4-08-09-2020}
|\Delta^w[\del_{\alpha}\del^IL^J w]|\leq CC_1\vep (s/t)^2s^{-2+\delta} \leq CC_1\vep t^{-2+\delta}.
\end{equation}
Remark that it is integrable with respect to $t$. 

On the other hand, recall the definition, for $|I|+|J|\leq N-4$,
$$
|S^w[\del_{\alpha} \del^IL^J w]| \leq Cs|v\del v|_{N-4}.
$$
By \eqref{eq10-01-09-2020},
\begin{equation}\label{eq5-08-09-2020}
|S^w[\del_{\alpha} \del^IL^J w]|\leq C(C_1\vep)^2t^{-1}s^{CC_1\vep}\leq C(C_1\vep)^2t^{-1+CC_1\vep}.
\end{equation}

Now recall \eqref{eq1-29-05-2020}, for $|I|+|J|\le N-4$
$$
\aligned
|s\del_t\del_{\alpha}\del^IL^J w(t,x)|\leq& CC_0\vep + CC_1\vep\int_2^t\tau^{-2+\delta}d\tau + C(C_1\vep)^2\int_2^t\tau^{-1+CC_1\vep}d\tau
\\
\leq& CC_1\vep t^{CC_1\vep}\leq CC_1\vep s^{CC_1\vep} .
\endaligned
$$
This concludes \eqref{eq18-01-09-2020}.
\paragraph*{Proof of \eqref{eq11-01-09-2020}.} This is based on Proposition \ref{lem1'-01-09-2020}. We write the Klein-Gordon equation of \eqref{eq-main} into the form of \eqref{eq'14-01-09-2020}:
\begin{equation}\label{eq7-08-09-2020}
\Box v + c^2\Big(1 - \underbrace{c^{-2}\big(B^{\alpha}\del_{\alpha}w_0 + Kw_0 + A^{\alpha\beta}\del_{\alpha}\del_\beta w\big)}_{\omega}\Big) v = 0.
\end{equation}
Following the notation of Proposition \ref{lem1'-01-09-2020}, $f = 0$ and $\omega$ is defined by the above expression. We remark that
$$
\aligned
&|\del \omega|\leq C\big(|\del w_0| + |\del\del w_0| + |\del^3 w|\big),
\\
&|\dels \omega|\leq C\big(|\dels w_0| + |\del\dels w_0| + |\del\del\dels w|\big).
\endaligned
$$
By \eqref{eq8-08-09-2020}, \eqref{eq9-08-09-2020}, \eqref{eq13-15-08-2020} and the following observation:
\begin{equation}\label{eq1-09-09-2020}
|\del\dels w|_{N-3}\leq Ct^{-1}|\del w|_{N-2}\leq CC_1\vep (s/t)s^{-2+\delta},
\end{equation}
one has
\begin{equation}\label{eq2-09-09-2020}
|\del \omega|\leq CC_1\vep (s/t)^{-1}s^{-2 + 2\delta},\quad |\dels \omega| \leq CC_1\vep (s/t)s^{-2+\delta}.
\end{equation}
The key is that both are integrable with respect to $s$ if we omit the conical decay. Then following the notation of Proposition \ref{lem1'-01-09-2020}, when $0\leq r/t\leq 3/5$
$$
s|v|(t,x) + s|((s/t)\del_t + (x^a/s)\delu_a) v|(t,x)\leq CC_0\vep + CC_1\vep\int_2^s \lambda^{-2+3\delta}d\lambda \leq CC_1\vep.
$$
When $3/5\leq r<1$, we need to apply \eqref{eq16-15-08-2020},
$$
\aligned
&s|v|(t,x) + s|((s/t)\del_t + (x^a/s)\delu_a) v|(t,x)
\\
\leq& CC_1\vep (s/t)^2\int_{\lambda_0}^s\lambda^{-2+\delta}d\lambda
+ C(C_1\vep)^2(s/t)^{-1}\int_{\lambda_0}^s(s/t)^2\lambda^{\delta}\ (s/t)\lambda^{-2+2\delta}d\lambda 
\\
\leq& CC_1\vep (s/t)^2\lambda_0^{-1+2\delta}.
\endaligned
$$
Remark that $\lambda_0\simeq (s/t)^{-1}$, we obtain:
$$
|v|(t,x) + |((s/t)\del_t + (x^a/s)\delu_a) v|(t,x) \leq CC_1\vep (s/t)^{3-2\delta}s^{-1}
$$
which gives the bound on $|v|$. Furthermore,
$$
(s/t)|\del_t v|(t,x)\leq CC_1\vep(s/t)^{3-2\delta}s^{-1} + CC_1\vep(s/t)^{-1}t^{-1}|v|_{1,1}\leq CC_1\vep (s/t)^2s^{-1}
$$
which show the bound on $|\del_t v|$. Recall that
$$
|\del_a v| = |t^{-1}L_av - (x^a/t)\del_t v|\leq Ct^{-1}|v|_{1,1} + |\del_tv|\leq CC_1\vep (s/t)s^{-1}.
$$ 
This leads to \eqref{eq11-01-09-2020}.

\paragraph*{Proof of \eqref{eq13-01-09-2020}.} This is the most critical one. We rely on  Proposition \ref{prpo2 wave-sharp}. Remark that
\begin{equation}\label{eq3-09-09-2020}
\Box \del_{\alpha} w = 2v\del_{\alpha}v.
\end{equation}
Then \eqref{eq4-08-09-2020} is still applicable. Furthermore, based on \eqref{eq11-01-09-2020}, 
\begin{equation}\label{eq4-09-09-2020}
|S^w[\del_{\alpha} w]|\leq C(C_1\vep)^2 (s/t)^3s^{-1} = C (C_1\vep)^2 (s/t)^2 t^{-1}.
\end{equation}
Here the conical factor $(s/t)^2$ in the above bound is crucial. This bound permits us to compare $S^w[\del_{\alpha}w]$ with $P(t,r)$ in Proposition \ref{prpo2 wave-sharp} and prevents a logarithmic loss.

Substitute \eqref{eq4-08-09-2020} and \eqref{eq4-09-09-2020} into \eqref{eq1-29-05-2020} and consider a point $(\bar{t},\bar{x})\in \Hcal_{\bar{s}}$,
$$
\aligned
&\bar{s}|\del_t\del_{\alpha} w|(\bar{t},\bar{x})
\\
\leq& CC_0\vep 
+ C(C_1\vep)^2\int_2^{\bar{t}} (s/t)^2 t^{-1}\big|_{\gamma(t;\bar{t},\bar{x})} e^{-\int_{t}^{\bar{t}}P_{\bar{t},\bar{x}}(\eta)d\eta} + CC_1\vep \int_2^{\bar{t}} t^{-2+\delta}dt
\\
\leq& CC_1\vep + C(C_1\vep)^2\int_2^{\bar{t}}(s/t)^2t^{-1}\big|_{\gamma(t;\bar{t},\bar{x})}e^{-\int_{t}^{\bar{t}}(s/t)^2t^{-1}\big|_{\gamma(\eta;\bar{t},\bar{x})}d\eta}dt
\\
\leq& CC_1\vep. 
\endaligned
$$
For the rest components of Hessian form, say $|\del_a\del_\alpha u|$, we have
 $$
\aligned
s|\del_a\del_{\alpha} w|(t,x)&=s |(\delu_a-(x^a/t)\del_t)\del_\alpha w|
                             \leq s|\delu_a\del_\alpha w|+s|\del_t\del_\alpha w|\\
                             &\leq s|\delu_a\del_\alpha w|+CC_1\vep
                             \leq (s/t)|L_a\del_\alpha w|+CC_1\vep\\
                             &\leq CC_1\vep.
\endaligned
$$ 
This concludes \eqref{eq13-01-09-2020}.

\subsection{Improving the energy bounds}
\label{subsec-model-improved}
Equipped with sharp bounds \eqref{eq12-01-09-2020} - \eqref{eq13-01-09-2020}, we are ready to improve \eqref{eq8-15-08-2020}.  
\paragraph*{Energy estimate on wave component.} We remark that
\begin{equation}\label{eq1-11-09-2020}
\aligned
\||v^2|_p\|_{L^2(\Hcal_s)} + \||v\del v|_p\|_{L^2(\Hcal_s)}\leq& CC_1\vep s^{-1}\Ecal_{0,c}^p(s,v)^{1/2} 
\\
&+ CC_1\vep s^{-1+CC_1\vep}\Ecal_{0,c}^{p-1}(s,v)^{1/2}.
\endaligned
\end{equation}
This is direct by \eqref{eq10-01-09-2020} and \eqref{eq11-01-09-2020}. 
$$
\aligned
&\||v\del v|_p\|_{L^2(\Hcal_s)}
\\
\leq& C\||v||\del v|_p\|_{L^2(\Hcal_s)} + \||\del v||v|_p\|_{L^2(\Hcal_s)} 
\\
&+ \||v|_{N-4}|\del v|_{p-1}\|_{L^2(\Hcal_s)} + \||v|_{p-1}|\del v|_{N-4}\|_{L^2(\Hcal_s)}
\\
\leq& CC_1\vep s^{-1}\|(s/t)|\del v|_p\|_{L^2(\Hcal_s)} + CC_1\vep s^{-1}\||v|_p\|_{L^2(\Hcal_s)}
\\
& + CC_1\vep s^{-1+CC_1\vep}\|(s/t)|\del v|_{p-1}\|_{L^2(\Hcal_s)}
  + CC_1\vep s^{-1+CC_1\vep}\||v|_{p-1}\|_{L^2(\Hcal_s)}
\endaligned
$$
which leads to the bound on $v\del v$. The bound on $v^2$ is similar. We omit the detail. 

Then recall Proposition \ref{prop 1 energy}, and apply it on to:
$$
\Box \del^IL^J w = \del^IL^J (v^2)
$$
and
$$
\Box \del^IL^J\del_{\alpha}u = \del^IL^J\del_{\alpha}(v^2)
$$
with $|I|+|J|\leq p$. We obtain:
$$
\aligned
&E_0(s,\del^IL^J w)^{1/2} + \sum_{\alpha} E_0(s,\del^IL^J\del_{\alpha}w)^{1/2} 
\\
\leq&  E_0(2,\del^IL^J w)^{1/2} + \sum_{\alpha} E_0(2,\del^IL^J\del_{\alpha}w)^{1/2}
\\ 
&+ C\int_2^s\|\del^IL^J(v^2)\|_{L^2(\Hcal_\tau)}d\tau + C\sum_{\alpha}\int_2^s\|\del^IL^J\del_{\alpha}(v^2)\|_{L^2(\Hcal_\tau)}d\tau
\\
\leq& CC_1\vep s^{-1}\Ecal_0^p(s,v)^{1/2} + CC_1\vep s^{-1+CC_1\vep}\Ecal_0^{p-1}(s,v)^{1/2}.
\endaligned
$$
So we conclude that
\begin{equation}\label{eq2-11-09-2020}
\aligned
&\Ecal_0^p(s,w)^{1/2} + \sum_{\alpha}\Ecal_0^p(s,\del_{\alpha} w)^{1/2} 
\\
\leq& \Ecal_0^p(2,w)^{1/2} + \sum_{\alpha}\Ecal_0^p(2,\del_{\alpha} w)^{1/2}
+ CC_1\vep \int_2^s\tau^{-1}\Ecal_{0,c}^p(\tau, v)^{1/2}d\tau
\\
&+ CC_1\vep\int_2^s\tau^{-1+CC_1\vep}\Ecal_{0,c}^{p-1}(\tau,v)^{1/2}d\tau.
\endaligned
\end{equation}
Then, write \eqref{eq6-01-09-2020} into the following form ($C_0\leq C_1$ and $C_1\vep\leq1$):
\begin{equation}\label{eq2-09-10-2020}
\Ecal_{0,c}^{N-2}(s,v)^{1/2}\leq CC_1\vep s^{CC_1\vep}\leq CC_1\vep s^{C(C_1\vep)^{1/2}}
\end{equation}
and substitute it into the above expression, we obtain:
\begin{equation}\label{eq3-09-10-2020}
\Ecal_0^{N-2}(s,w)^{1/2} + \sum_{\alpha}\Ecal_0^p(s,\del_{\alpha} w)^{1/2}  \leq C_0\vep +  C(C_1\vep)^{3/2}s^{C(C_1\vep)^{1/2}}.
\end{equation}

\paragraph*{Energy estimate on Klein-Gordon component.} 
This is also by Proposition \ref{prop 1 energy}. We will establish
\begin{equation}\label{eq3-11-09-2020}
\aligned
\||w_0 v|_p\|_{L^2(\Hcal_s)} +& \||v\del w_0|_p\|_{L^2(\Hcal_s)} + \||v\del\del w|_p\|_{L^2(\Hcal_s)} 
\\
\leq& 
CC_1\vep s^{-1}\Ecal_{0,c}^p(s,v)^{1/2} + CC_1\vep s^{-1}\sum_{\alpha}\Ecal_0^p(s,\del_{\alpha} w)^{1/2} 
\\
&+ C(C_1\vep)^{5/3}s^{-1+C(C_1\vep)^{1/3}}.
\endaligned
\end{equation}
For the first term,
$$
\aligned
&\||w_0v|_p\|_{L^2(\Hcal_s)}
\\
\leq& C\||w_0|_{N-2}|v|_p\|_{L^2(\Hcal_s)} + C\||v|_{N-4}|w_0|_{p-1}\|_{L^2(\Hcal_s)} 
 + C\||v||w_0|_p\|_{L^2(\Hcal_s)}
\\
\leq& CC_1\vep s^{-1}\||v|_p\|_{L^2(\Hcal_s)} + CC_1\vep s^{-1+CC_1\vep}\|(s/t)|w_0|_{p-1}\|_{L^2(\Hcal_s)} 
\\
&+ CC_1\vep s^{-1}\|(s/t)|w_0|_p\|_{L^2(\Hcal_s)}
\\
\leq& CC_1\vep s^{-1}\Ecal_{0,c}^p(s,v)^{1/2} + CC_1\vep s^{-1} \Fcal_2^p(2;s,w_0)
\\
 &+ CC_1\vep s^{-1+CC_1\vep}\Fcal_2^{p-1}(2;s,w_0).
\endaligned
$$
Here $\|(s/t)|w_0|_p\|_{L^2(\Hcal_s)}$ is bounded by $\Fcal_2^p(2;s, w_0)$. The latter is bounded as (recalling \eqref{eq1-09-10-2020}),
$$
\aligned
\Fcal_2^p(2;s,w_0)\leq& \|(s/t)|w_0|_p\|_{L^2(\Hcal_s)} + \Ecal_2^p(s,w_0)^{1/2} + \int_2^s\tau^{-1}\Ecal_2^p(\tau,w_0)^{1/2}d\tau
\\
\leq& CC_0\vep + CC_1\vep \int_2^s\tau^{-1 + (C_1\vep)^{1/3}} d\tau\leq C(C_1\vep)^{2/3} s^{(C_1\vep)^{1/3}}.
\endaligned
$$
Then
$$
\||w_0v|_p\|_{L^2(\Hcal_s)}\leq CC_1\vep s^{-1}\Ecal_{0,c}^p(s,v)^{1/2} + C(C_1\vep)^{5/3}s^{-1+C(C_1\vep)^{1/3}.}
$$

The second term is bounded as following:
$$
\aligned
\||v\del w_0|_p\|_{L^2(\Hcal_s)}\leq& C\||\del w_0|_{N-3}|v|_p\|_{L^2(\Hcal_s)} + C\||v|_{N-4}|\del w_0|_p\|_{L^2(\Hcal_s)} 
\\
\leq& CC_1\vep s^{-1}\Ecal_{0,c}^p(s,v)^{1/2} + CC_1\vep s^{-1+CC_1\vep}\|(s/t)|\del w_0|_p\|_{L^2(\Hcal_s)}
\\
\leq& CC_1\vep s^{-1}\Ecal_{0,c}^p(s,v)^{1/2} + C(C_1\vep)^{5/3}s^{-1+C(C_1\vep)^{1/3}}.
\endaligned
$$

For the last term,
$$
\aligned
&\||v\del\del w|_p\|_{L^2(\Hcal_s)}
\\
\leq& C\||v||\del\del w|_p\|_{L^2(\Hcal_s)} + C\||\del\del w||v|_p\|_{L^2(\Hcal_s)}
\\
& + C\||v|_{N-3}|\del\del w|_{N-1}\|_{L^2(\Hcal_s)} + C\||v|_{N-1}|\del\del w|_{N-3}\|_{L^2(\Hcal_s)}
\\
\leq& CC_1\vep s^{-1} \|(s/t)|\del\del w|_p\|_{L^2(\Hcal_s)} + CC_1\vep s^{-1}\||v|_p\|_{L^2(\Hcal_s)}
\\
& + CC_1\vep s^{-2+\delta}\|(s/t)^2s|\del\del w|_{N-1}\|_{L^2(\Hcal_s)} 
+ CC_1\vep s^{-2+\delta}\|(s/t)^{-1}|v|_{N-1}\|_{L^2(\Hcal_s)}
\\
\leq& CC_1\vep s^{-1}\Big(\sum_{\alpha}\Ecal_0^p(s,\del_{\alpha}w)^{1/2} + \Ecal_{0,c}^p(s,v)^{1/2}\Big)
 + C(C_1\vep)^2s^{-2+3\delta}.
\endaligned
$$
Here for the second inequality we have applied \eqref{eq11-01-09-2020}, \eqref{eq13-01-09-2020}, \eqref{eq16-15-08-2020} and \eqref{eq13-15-08-2020}. For the last inequality, \eqref{eq12-15-08-2020} and \eqref{eq17-15-08-2020} are applied.

Substitute these bounds into \eqref{ineq 3 prop 1 energy},
$$
\aligned
\Ecal_{0,c}^p(s,v)^{1/2}\leq&\Ecal_{0,c}^p(2,v)^{1/2} 
\\
&+ C\int_2^s\big(\||v\del w_0|_p\|_{L^2(\Hcal_\tau)} + \||vw_0|_p\|_{L^2(\Hcal_\tau)} + \||v\del\del w|_p\|_{L^2(\Hcal_\tau)}d\tau
\endaligned
$$
which leads to
\begin{equation}\label{eq4-11-09-2020}
\aligned
\Ecal_{0,c}^p(s,v)^{1/2}\leq& \Ecal_{0,c}^p(2,v)^{1/2} + C(C_1\vep)^{4/3}s^{C(C_1\vep)^{1/3}} 
\\
&+ CC_1\vep\int_2^s\tau^{-1}\Big(\sum_{\alpha}\Ecal_0^p(s,\del_{\alpha}w)^{1/2} + \Ecal_{0,c}^p(s,v)^{1/2}\Big)d\tau.
\endaligned
\end{equation}
\paragraph*{Inductional argument.}
For convenience, we denote by
$$
A^p(s) := \max\Big\{\sum_{\alpha}\Ecal_0^p(s,\del_{\alpha} w)^{1/2}, \Ecal_0^p(s,w)^{1/2}, \Ecal_{0,c}^p(s,v)^{1/2} \Big\}.
$$
Then \eqref{eq2-11-09-2020} and \eqref{eq4-11-09-2020} lead to
\begin{equation}\label{eq5-11-09-2020}
\aligned
A^p(s) \leq& A^p(2) + C(C_1\vep)^{4/3} s^{C(C_1\vep)^{1/3}}
\\
&+ CC_1\vep \int_2^s\tau^{-1}A^p(\tau)d\tau + CC_1\vep\int_2^s\tau^{-1+CC_1\vep}A^{p-1}(\tau)d\tau
\endaligned
\end{equation}
for $0\leq p\leq N$. Recall that \eqref{eq2-09-10-2020} and \eqref{eq3-09-10-2020} shows that
\begin{equation}
A^{N-2}(s)\leq CC_1\vep s^{C(C_1\vep)^{1/2}}.
\end{equation}

Now we concentrate on the case $p=N-1$. 
$$
\aligned
A^{N-1}(s)\leq& C_0\vep + C(C_1\vep)^{4/3}s^{C(C_1\vep)^{1/3}} + C(C_1\vep)^2\int_2^s\tau^{-1}A^{N-1}(\tau)d\tau
\\
&+ CC_1\vep\int_2^s\tau^{-1+CC_1\vep}A^{N-2}d\tau
\\
\leq& C_0\vep + C(C_1\vep)^{4/3}s^{C(C_1\vep)^{1/3}} +  C(C_1\vep)\int_2^s\tau^{-1}A^{N-1}(\tau)d\tau.
\endaligned
$$
Gronwall's inequality leads to
\begin{equation}\label{eq4-09-10-2020}
A^{N-1}(s)\leq C_0\vep(s/2)^{CC_1\vep} + C(C_1\vep)^{4/3}s^{C(C_1\vep)^{1/3}}\leq CC_1\vep s^{C(C_1\vep)^{1/3}}.
\end{equation}

Then taking $p=N$ and the above bound on $A^{N-1}$, do again the above argument and we obtain:
\begin{equation}
A^N(s)\leq C_0\vep(s/2)^{CC_1\vep} + C(C_1\vep)^{4/3}s^{C(C_1\vep)^{1/3}}.
\end{equation}

\subsection{Conclusion of bootstrap argument.}
Taking
\begin{equation}\label{eq12-09-10-2020}
C_1>2C_0,\quad \vep\leq \delta^3/(C^3C_1),\quad \vep \leq \frac{(C_1-2C_0)^3}{8C^3C_1^4},\quad \vep\leq \frac{c^2}{2CC_1},
\end{equation}
one guarantees the following bounds:
$$
C(C_1\vep)\leq \frac{c^2}{2},
\quad
C(C_1\vep)^{1/3}\leq \delta, 
\quad
C_0\vep + C(C_1\vep)^{4/3}\leq \frac{1}{2}C_1\vep.
$$
So  we guarantees \eqref{eq7-08-10-2020} and
$$
A^N(s) = \max\Big\{\Ecal_0^N(s,w)^{1/2}, \Ecal_{0,c}^N(s,v)^{1/2}, \sum_{\alpha}\Ecal_0^N(s,\del_{\alpha} w)^{1/2}\Big\}\leq \frac{1}{2}C_1\vep s^{\delta}.
$$
This closes the bootstrap argument.

\subsection{Application on Klein-Gordon-Zakharov system}
Clearly, \eqref{eq1-Zakharov} is in the form of \eqref{eq1b-main}. So Theorem \ref{thm-main} applies directly and we conclude the global existence result for \eqref{eq1-Zakharov} with small localized regular initial data.

\section{Return to the totally geodesic wave map system}\label{sec-conclusion-wave-map}
\subsection{The stability problem of a type of totally geodesic wave maps}\label{subsec-wave-map}
A detailed explanation and formulation can be found in \cite{Ab-2019}. Here we only give an outline.

Let $\RR^{2+1}$ be the standard $2+1$ dimensional Minkowski space-time with signature $(-,+,+)$. Let $(M,g)$ be a $n-$dimensional space-form. Consider a map
$\RR^{2+1}\stackrel{\phi}{\longrightarrow}M$. This map is called {\sl wave map} if it is a critical point of the following action:  
\begin{equation}
S[\phi] = \int_{\RR^{2+1}}\langle d\phi, d\phi \rangle_{T^*\RR^{2+1}\otimes \phi^{-1}TM}\, d\text{vol}_{\RR^{2+1}}.
\end{equation}
Then $\phi$ satisfies the following Euler-Lagrangian equation:
\begin{equation}\label{eq1-04-10-2020}
\Box_m \phi^i + \Gamma_{jk}^i(\phi) \m^{\mu\nu}\del_{\mu}\phi^j\del_{\nu}\phi^k = 0
\end{equation}
where $\m$ is the Minkowski metric defined on $\RR^{2+1}$. $\Box_\m = m^{\alpha\beta}\del_{\alpha}\del_{\beta} = -\del_t^2 + \Delta_{\RR^2}$. $\Gamma_{ij}^k(\phi)$ are the Christoffel symbols of $(M,g)$ evaluated along the image of $\phi$. 

We consider a wave map from $\RR^{2+1}$ to $(M,g)$ with the following factorization:
$$
\varphi:\RR^{2+1}\stackrel{\varphi_S}{\longrightarrow}\RR\stackrel{\varphi_I}{\longrightarrow} M.
$$
Here $\varphi_S$ is a semi-Riemannian submersion to either $(\RR,e)$ or $(\RR,-e)$ and $\varphi_I$ is a immersion from $(\RR,e)$ to $(M,g)$. By \cite{ES64} and \cite{Vil70}, the above factorization implies that $\varphi = \varphi_I\circ\varphi_S$ is totally geodesic. Then one consider the stability problem of $\varphi$. Furthermore, regarding \cite{Vil70}, $\varphi_S$ is prescribed to be a linear function $\RR^{2+1}\rightarrow \RR$ with $m(d\varphi_S,d\varphi_S) = \pm 1$. Then $\varphi_I$ is a immersed geodesic on $(M,g)$.

The	quantitative formulation and analysis of this problem is based on the geodesic normal coordinates.  It permits one to parameterize a tubular neighborhood of an arbitrary geodesic, in which the Christoffel symbols vanish along the geodesic.  Let us give a brief description. Let $(M,g)$ be a complete Riemannian manifold and $\gamma:\RR\rightarrow M$ be a fixed geodesic. We parameterize it with arc-length. At $\gamma(0)$, let $\vec{e}_1 = \dot\gamma(0)$ and
$$
e^{\perp} := (\vec{e}_2,\cdots, \vec{e}_n),\quad \vec{e}_i\perp \vec{e}_j,\quad |\vec{e}_i|=1.
$$
For $x^1\in \RR$, define $\vec{e}_i$ by parallel transporting along $\gamma$. This forms an normalized orthogonal frame along $\gamma$.  Let $\exp_{\gamma(x^1)}(t\vec{v})$ be the geodesic satisfying
$$
\frac{d}{dt}\exp_{\gamma(x^1)}(t\vec{v})\big|_{t=0} = \vec{v},\quad \exp_{\gamma(x^1)}(t\vec{v})\big|_{t=0} = \gamma(x^1).
$$ 
with $\vec{v}\in \dot\gamma(x^1)^{\perp}$ and for $(x^1,\bar{x}) = (x^1,x^2,\cdots x^n)$ with $|\bar{x}|$ sufficiently small,
$$
\sigma: (x^1,\bar{x})\rightarrow \exp_{\gamma(x^1)}\Big(\sum_{j=2}^nx^j\vec{e}_j\Big)
$$
gives a parameterization of the tubular neighborhood of $\gamma$. This is called the geodesic normal coordinates. Due to the assumption that $(M,g)$ being a space-form, this coordinate system is well defined in $(-\infty,\infty)\times \{|\bar{x}|< \delta\}$ with a fixed $\delta>0$ (which is the {\sl focal radius}).

With this geodesic normal coordinates, a perturbation of $\varphi$ is described as following (see in detail \cite{Ab-2019}). We construct the above geodesic normal coordinates in a tubular neighborhood of $\varphi_I(\RR)$.  Then $\varphi$ is written as
$$
\RR^{2+1}\stackrel{\varphi_S}{\longrightarrow}\RR\stackrel{\varphi_I}{\longrightarrow} M,\quad (t,x)\rightarrow \sigma\big(\varphi_S(t,x),0\big).
$$
Then we perturb $\varphi$ as following, consider 
$$
\tilde{\varphi} : \RR^{2+1}\rightarrow M,\quad (t,x)\rightarrow \sigma\big(\varphi_S(t,x) + \phi^1(t,x),\phi^k(t,x)\big), \quad k=2,\cdots, n
$$ 
and demand that $\tilde{\varphi}$ is again a wave map. This leads to, thanks to \eqref{eq1-04-10-2020} and the fact that $\varphi_S$ being linear,
\begin{equation}\label{eq15-07-10-2020}
\begin{aligned}
&\Box_m \phi^1 + \Gamma_{jk}^1(\varphi_S+\phi^1,\phi^k)\cdot \m(\bar{\phi}^j,\bar{\phi}^k) = 0,
\\
&\Box_m \phi^i + \Gamma_{jk}^i(\varphi_S+\phi^1,\phi^k)\cdot \m(\bar{\phi}^j,\bar{\phi}^k) = 0,
\quad  i = 2,\cdots, n
\end{aligned}
\end{equation}
\footnote{Here $\Box_m := m^{\alpha\beta}\nabla_{\alpha}\nabla_{\beta} = -\del_t^2-\sum_a\del_a^2$.} with $\bar{\phi}^1 = \varphi_S(t,x) + \phi^1, \bar{\phi}^k = \phi^k, k=2,3,\cdots n$.

Then  one develops the nonlinear terms $\Gamma_{jk}^1(\varphi_S+\phi^1,\phi^k)\cdot \m(\bar{\phi}^j,\bar{\phi}^k)$ into Taylor series at each point of $\varphi_I(\RR)$, i.e., at $(\phi^1,\phi^k) = 0$. Since (due to the construction of the geodesic normal coordinates) $\Gamma_{jk}^i\equiv 0$ along $\varphi_I(\RR)$, 
\begin{equation}\label{eq16-07-10-2020}
\del_1^q\Gamma_{jk}^i(\varphi_S,0) = 0,\quad q=0,1,2,\cdots
\end{equation}
This leads to the fact that in the development of $\Gamma$ there is no monominal containing $\phi^1$.
Furthermore,
$$
\aligned
\m(d\bar{\phi}^j,d\bar{\phi}^k) 
=&\m(d\varphi_S,d\varphi_S)\delta_1^i\delta_1^j
+ \m(d\varphi_S,d\phi^k)\delta_1^i
+ \m(d\phi^j,d\varphi_S)\delta_1^k
+ \m(d\phi^j,d\phi^k)
\\
=& \pm \delta_1^i\delta_1^j
+ \m(d\varphi_S,d\phi^k)\delta_1^i
+ \m(d\phi^j,d\varphi_S)\delta_1^k
+ \m(d\phi^j,d\phi^k).
\endaligned
$$
Then following the procedure in Section 3 of \cite{Ab-2019},  $(\phi^1,\phi^k)$ satisfies the following system \footnote{Recall that we have taken $\Box  = \del_t^2 - \sum_a\del_a^2$.} when $(M,g)$ is of sectional curvature $\equiv -1$:
\begin{subequations}\label{eq2-04-10-2020}
	\begin{equation}\label{eq2a-04-10-2020}
	\aligned
	-\Box \phi^1 = & -2\sum_{k=2}^n\phi^k\del_t\phi^k - 2\sum_{k,l=2}^n\del_k\del_l\Gamma_{j1}^1(\varphi_S,0)\phi^k\phi^l\cdot \m(d \phi^j,d\varphi_S)
	\\
	& + \sum_{k=2}^n\del_k\Gamma_{jl}^1(\varphi_S,0)\phi^k\cdot \m(d\phi^j,d\phi^l)
	+ \sum_{k,j,l=2}^n\del_k\del_j\del_l\Gamma_{11}^1(\varphi_S,0)\phi^k\phi^j\phi^l
	\\
	& +\text{ h.o.t.}
	\\
	-\Box \phi^k - \phi^k =& 2\phi^k\del_t \phi^1 
	- 2\sum_{k,l=2}^n\del_k\del_l\Gamma_{j1}^i(\varphi_S,0)\phi^k\phi^l\cdot \m(d\phi^j,d\varphi_S)
	\\
	& + \sum_{k=2}^n\del_k\Gamma_{jl}^1(\varphi_S,0)\phi^k\cdot \m(d\phi^j,d\phi^l) 
	- \sum_{k,j,l=2}^n\del_k\del_j\del_l \Gamma_{11}^i(\varphi_S,0)\phi^k\phi^j\phi^l 
	\\
	&+ \text{h.o.t.}
	\endaligned
	\end{equation}
	And when $(M,g)$ is of sectional curvature $\equiv 1$,
	\begin{equation}\label{eq2b-04-10-2020}
	\aligned
	-\Box \phi^1 = & 2\sum_{k=2}^n\phi^k\del_1\phi^k - 2\sum_{k,l=2}^n\del_k\del_l\Gamma_{j1}^1(\varphi_S,0)\phi^k\phi^l\cdot \m(d \phi^j,d\varphi_S)
	\\
	& + \sum_{k=2}^n\del_k\Gamma_{jl}^1(\varphi_S,0)\phi^k\cdot \m(d\phi^j,d\phi^l)
	- \sum_{k,j,l=2}^n\del_k\del_j\del_l\Gamma_{11}^1(\varphi_S,0)\phi^k\phi^j\phi^l 
	\\
	&+ \text{h.o.t.}
	\\
	-\Box \phi^k - \phi^k =& -2\phi^k\del_1 \phit^1
	- 2\sum_{k,l=2}^n\del_k\del_l\Gamma_{j1}^i(\varphi_S,0)\phi^k\phi^l\cdot \m(d\phi^j,d\varphi_S)
	\\
	& + \sum_{k=2}^n\del_k\Gamma_{jl}^1(\varphi_S,0)\phi^k\cdot \m(d\phi^j,d\phi^l) 
	- \sum_{k,j,l=2}^n\del_k\del_j\del_l \Gamma_{11}^i(\varphi_S,0)\phi^k\phi^j\phi^l 
	\\
	&+ \text{h.o.t.}
	\endaligned
	\end{equation}
\end{subequations}

We summarize the key structures of the above two systems. Firstly, in both cases, the quadratic terms coupled in wave equation are in divergence form. Secondly, as consequence of Lemma 2.5 of \cite{Ab-2019}, the coefficients $\del^k\Gamma(\varphi_S,0)$ can be regarded as universal constants. Remark that in order to guarantee global existence in $\RR^{2+1}$, we must also analyze the cubic terms (this is explained in \cite{A2} in pure wave case). In \eqref{eq2-04-10-2020} we are sufficiently lucky such that in both cases and both wave and Klein-Gordon equations, the cubic terms are either null cubic forms or containing at least two Klein-Gordon factors. Finally, the higher order terms can be written as linear combinations of
\begin{equation}\label{eq18-07-10-2020}
\aligned
&\phi^j\phi^k \m(d\phi^a,d\phi^b)\cdot O(\phi),\quad \phi^j\phi^k\phi^l\del\phi^c\cdot O(\phi), 
\\
&1\leq a,b,c \leq n,\quad  2\leq j,k,l \leq n
\endaligned
\end{equation}
with coefficients $\del^q\Gamma(\varphi_S,0)$ which can be regarded as universal constants due the the Lemma 2.5 of \cite{Ab-2019} and the remark made below equation (5.2) therein. The important structure is the two Klein-Gordon factors. This is due to \eqref{eq16-07-10-2020}. 

\subsection{Formulation of the auxiliary system and statement of main result}
This subsection is devoted to the construction of the auxiliary system to \eqref{eq2-04-10-2020}.  We will only regard the case of negative sectional curvature. The positive case is similar, we omit the detail. Firstly, we write \eqref{eq2a-04-10-2020} into the following form:
\begin{equation}\label{eq1-02-09-2020}
\aligned
&\Box \phi = -2\sum_{k=2}^n\phi^k \del_t\phi^k + S_W[\phi],
\\
&\Box \phi^k + \phi^k = 2\phi^k \del_t\phi^1 + S_{GK}^{k}[\phi] , \quad 2\leq k\leq n,
\\
&\phi^1(2,x) = \phi^1_0(x),\quad \del_t \phi^1(2,x) = \phi^1_1(x),
\\
&\phi^k(2,x) = \phi^k_0(x),\quad \del_t \phi^k(2,x) = \phi^k_1(x),\quad 2\leq k\leq n.
\endaligned
\end{equation}
Here $S_W$ and $S^k_{KG}$ contains the third and higher order terms. By introducing the shifted primitive of $\phi^1$ up to second order:
\begin{equation}\label{eq2-05-10-2020}
\phi^1 = \del_t w + w_0,
\end{equation}
we arrive at the following auxiliary system:
\begin{equation}\label{eq4-02-09-2020}
\left\{
\aligned
&\Box w = - \sum_{k=2}^n|\tilde{\phi}^k|^2,
\\
&\Box w_0 = S_W[(\del_tw + w_0),\tilde{\phi}^k],
\\
&\Box \tilde{\phi}^k + \tilde{\phi}^k = 2\tilde{\phi}^k\del_t\big(\del_t w + w_0\big) + S_{KG}^{k}[(\del_t w + w_0),\tilde{\phi}^k],\quad 2\leq k\leq n
\endaligned
\right.
\end{equation}
with initial data
\begin{equation}\label{eq5-02-09-2020}
\aligned
&w (2,x) = 0,\quad \del_t w(2,x) = 0,\quad \tilde{\phi}^k(2,x) = \phi^k_0(x),\quad \del_t\tilde{\phi}^k(2,x) = \phi^k_1(x)
\\
& w_0(2,x) = u_0(x), \quad \del_t w_0(2,x) = u_1(x) + \sum_{k=2}^n|\tilde{\phi}^k_0(x)|^2.
\endaligned
\end{equation}
Parallel to Lemma \ref{lem1-05-10-2020}, the following result holds:
\begin{lemma}\label{lem2-05-10-2020}
Let $(w,w_0,\tilde{\phi}^k)$ be a $C^3$ solution to \eqref{eq4-02-09-2020}, then $(\phi^1,\phi^k)$ with $\phi^1$ defined by \eqref{eq2-05-10-2020} and $\phi^k = \phit^k$ is the $C^2$ solution to \eqref{eq1-02-09-2020}. 
\end{lemma}
\begin{remark}\label{rk-1-12-10-2020}
Compare \eqref{eq4-02-09-2020} with \eqref{eq-main}, the main difference is that $w_0$ is no longer a solution to free linear wave equation. However it is not so far from that because the right-hand-side of equation of $w_0$ in \eqref{eq4-02-09-2020} is {\sl cubic}. Another important difference is that in \eqref{eq4-02-09-2020}, $w_0$ is coupled only with its gradient. More precisely, the term $vw_0$ does not exits in Klein-Gordon equations. Although it is not necessary, this structure will simplify a lot our argument. For example we need to the bound the conformal energy of $w_0$, which was necessary in Section \ref{sec-bootstrap} in order to bound the $L^2$ and pointwise bounds of $|w_0|_{p}$. 
\end{remark}

Then we establish the following result:
\begin{theorem}\label{thm-wave-map}
Suppose that $\phi^j_i$, $i=1,2$ and $j=1,\cdots, n$ are compactly supported in $\{|x|>1\}$. Then there is a integer $N\geq 7$ and positive constant $\vep_0>0$ determined by the system and $N$, such that $\forall\, 0\leq \vep\leq \vep_0$, if
\begin{equation}\label{eq3-05-10-2020}
\|\phi^j_0\|_{H^{N+1}}\leq \vep,\quad \|\phi_1^j\|_{H^N}\leq \vep,\quad j=1,2,\cdots n,
\end{equation}
then the local solution to the Cauchy problem associated with \eqref{eq1-02-09-2020} together with the initial data \eqref{eq3-05-10-2020} extends to time infinity.
\end{theorem}
\begin{remark}
This regularity $N\geq 7$ is to guarantee that for $A,B$ sufficiently regular functions, 
$$
|AB|_p\leq |A|_p|B|_{N-4} + |A|_{N-4}|B|_p.
$$
\end{remark}

\begin{remark}
The restriction $N\geq 7$ is not optimal. As we will see in the proof, because the auxiliary system is {\bf subcritical} in the sens of principle decay, this regularity can probably be improved. However in the regime of Lemma \ref{lem2-05-10-2020} there is a limit. Regarding Lemma \ref{lem2-05-10-2020} and the auxiliary system \eqref{eq4-02-09-2020}, we need to guarantee the $C^3$ regularity of $w$ and $C^2$ regularity of $(w_0,\phit^k)$. So we need $H^4$ regularity on $\phi_1^j$ and $H^5$ regularity on $\phi^j_0$.
\end{remark}

\section{Proof of Theorem \ref{thm-wave-map}}
\label{sec-wave-maps-proof}
\subsection{Bootstrap assumption and direct bounds}
We establish this global stability result via \eqref{eq4-02-09-2020}.  This is quite similar to the Proof in Section \ref{sec-bootstrap}. In fact there is a one-to-one correspondence between the subsections here to those in Section \ref{sec-bootstrap}, except the Subsection \ref{subsec-high-order} in which we treat the high-order terms. There are also other modifications among which the most important is the bound on $w_0$. In this case one only demands a uniform bound on standard energy but not on conformal energy. The reason is explained in Remark \ref{rk-1-12-10-2020}. 

To get started, let
$$
\max\Big\{\sum_{\alpha}\Ecal_0^N(2,\del_{\alpha}w)^{1/2}, \Ecal_0^N(2,w)^{1/2},\sum_{k=2}^n\Ecal_{0,1}^N(2,\phit^k)^{1/2}, \Ecal_0^N(2,w_0)^{1/2}\Big\}= C_0\vep.
$$
Then we make the following bootstrap bound on $[2,s_1]$: 
\begin{equation}\label{eq4-05-10-2020}
\max\Big\{\sum_{\alpha}\Ecal_0^N(s,\del_{\alpha}w)^{1/2}, \Ecal_0^N(s,w)^{1/2},\sum_{k=2}^n\Ecal_{0,1}^N(s,\phit^k)^{1/2}\Big\}\leq C_1\vep s^{\delta}.
\end{equation}
Suppose furthermore that
\begin{equation}\label{eq6-05-10-2020}
\Ecal_0^N(s,w_0)^{1/2}\leq C_1\vep . 
\end{equation}
Here $1/100\leq \delta\leq 1/20$.  We will prove the following {\sl improved energy bounds} on the same time interval:
\begin{equation}\label{eq7-05-10-2020}
\max\Big\{\sum_a\Ecal_0^N(s,\del_{\alpha}w)^{1/2}, \Ecal_0^N(s,w)^{1/2},\sum_{k=2}^n\Ecal_{0,1}^N(s,\phit^k)^{1/2}\Big\}\leq \frac{1}{2}C_1\vep s^{\delta}.
\end{equation}
\begin{equation}\label{eq8-05-10-2020}
\Ecal_0^N(s,w_0)^{1/2}\leq \frac{1}{2}C_1\vep. 
\end{equation}

By \eqref{eq3-10-06-2020}, the following decay are guaranteed by \eqref{eq4-05-10-2020}:
\begin{equation}\label{eq2-06-10-2020}
s|\del\del w|_{N-2} + t|\del\dels w|_{N-2}\leq CC_1\vep s^{\delta},
\end{equation}
\begin{equation}\label{eq5-05-10-2020}
s|\del w|_{N-2} + t|\dels w|_{N-2}\leq CC_1\vep s^{\delta}
\end{equation}
and
\begin{equation}\label{eq9-05-10-2020}
s|\del \phit^k|_{N-2} + t|\dels \phit^k|_{N-2} + t|\phit^k|_{N-2}\leq CC_1\vep s^{\delta}
\end{equation}
which leads to
\begin{equation}\label{eq10-05-10-2020}
t|\del \phit^k|_{N-3} + t^2|\dels \phit^k|_{N-3}\leq CC_1\vep s^{\delta}.
\end{equation}

By \eqref{eq7-10-06-2020} combined with \eqref{eq4-05-10-2020},
\begin{equation}\label{eq8-09-10-2020}
st|\del\dels w|_{N-3} + t^2|\dels\dels w|_{N-3}\leq CC_1\vep s^{\delta}.
\end{equation}
This leads to the following bound. For $|I|+|J|\leq N-3$, 
$$
|\del_r\delu_a\del^IL^J w|\leq CC_1\vep (s/t)s^{-2+\delta} \leq CC_1\vep (t-r)^{-1/2+\delta/2}t^{-3/2+\delta/2}.
$$
For a fixed $(t,x)\in \Hcal_{[2,s_1]}$, integrate this inequality along the segment $\{(t,\lambda x/|x|),|x|\leq \lambda\leq t-1\}$ and remark that $\delu_a\del^IL^J$ vanishes on $\del\Kcal = \{r=t-1\}$, we obtain:
$$
\aligned
|\delu_a\del^IL^J w(t,x)|\leq& \int_{|x|}^{t-1}|\del_r\delu_a\del^IL^J u|(t,\lambda x/|x|)d\lambda 
\\
\leq& CC_1\vep (t-r)^{1/2+\delta/2}t^{-3/2+\delta/2}
\leq CC_1\vep (s/t)^2s^{-1+\delta}.
\endaligned
$$
This leads to
\begin{equation}\label{eq4-07-10-2020}
|\dels w|_{N-3}\leq CC_1\vep (s/t)^2s^{-1+\delta}.
\end{equation}

Furthermore, by \eqref{eq6-05-10-2020} and \eqref{eq4-10-06-2020}
\begin{equation}\label{eq1-07-10-2020}
s|\del w_0|_{N-2} + t|\dels w_0|_{N-2} \leq CC_1\vep.
\end{equation}
Similar to \eqref{eq4-07-10-2020}, the following bound holds for $w_0$
\begin{equation}\label{eq10-08-10-2020}
|\dels w_0|_{N-3}\leq CC_1\vep(s/t)^2s^{-1}.
\end{equation}
Furthermore
\begin{equation}\label{eq10-09-10-2020}
|\del\dels w_0|_{N-3}\leq CC_1\vep (s/t)s^{-2}.
\end{equation}

Recall \eqref{eq2-05-10-2020}, the following bounds hold:
\begin{equation}\label{eq9-07-10-2020}
\|(s/t)|\del \phi^1|_N\|_{L^2(\Hcal_s)} + \||\dels \phi^1|_N\|_{L^2(\Hcal_s)}\leq CC_1\vep s^{\delta}.
\end{equation}
\begin{equation}\label{eq2-07-10-2020}
|\del \phi^1|_{N-2}\leq CC_1\vep s^{-1+\delta},
\end{equation}
\begin{equation}\label{eq3-07-10-2020}
|\dels \phi^1|_{N-2}\leq CC_1\vep (s/t)s^{-1+\delta}.
\end{equation}
Combining\eqref{eq8-09-10-2020} and \eqref{eq10-08-10-2020}, we obtain 
\begin{equation}\label{eq11-08-10-2020}
|\dels \phi^1|_{N-3}\leq CC_1\vep (s/t)s^{-2+\delta} + CC_1\vep (s/t)^2s^{-1} \leq CC_1\vep(s/t)^2s^{-1+\delta}.
\end{equation}

\subsection{Bounds on Hessian form of $w$}
In this subsection we establish the following bound:
\begin{equation}\label{eq5-07-10-2020}
(s/t)^2|\del\del w|_{N-3}\leq CC_1\vep (s/t)s^{-2+2\delta},
\end{equation}
\begin{equation}\label{eq6-07-10-2020}
\|(s/t)^2s|\del\del w|_{N-1}\|_{L^2(\Hcal_s)}\leq CC_1\vep s^{2\delta}.
\end{equation}
These are exactly the same to \eqref{eq12-15-08-2020} and \eqref{eq13-15-08-2020}. We establish bounds parallel to \eqref{eq2-08-10-2020} and \eqref{eq3-08-10-2020}. To do so, remark that
$$
\aligned
&|(\phit^k)^2|_{N-2}\leq C(C_1\vep)^2(s/t)^2s^{-2+2\delta},
\\
&\|(s/t)^{-1}|(\phit^k)^2|_N\|_{L^2(\Hcal_s)}\leq C(C_1\vep)^2s^{-1+2\delta}
\endaligned
$$
where \eqref{eq9-05-10-2020} is applied.

Recall the relation \eqref{eq2-05-10-2020}, a direct consequence of \eqref{eq5-07-10-2020} is
\begin{equation}\label{eq7-07-10-2020}
|\del\phi^1|_{N-3}\leq CC_1\vep(s/t)^{-1}s^{-2+2\delta} + CC_1\vep s^{-1}.
\end{equation}

\subsection{Bounds on $S_W[\phi]$ and $S_{KG}^k[\phi]$ and bounds on $w_0$}
\label{subsec-high-order}
This subsection is devoted to the high-order terms. We establish their bounds and the give two direct bounds on $w_0$.
\paragraph*{$L^2$ Bounds on higher-order terms.}
We firstly establish the following $L^2$ bounds:
\begin{equation}\label{eq1-06-10-2020}
\|(s/t)^{-1}|S_W[\phi^1,\phit^k]|_N\|_{L^2(\Hcal_s)} + \|(s/t)^{-1}|S_{KG}^k[\phi^1,\phit^k]|_N\|_{L^2(\Hcal_s)} \leq C(C_1\vep)^3s^{-2+3\delta}.
\end{equation}
$S_W, S_{KG}^k$ vanish in third order with respect to their argument.  Recall Lemma 2.5 of \cite{Ab-2019} and the remark made after (5.2) therein, the coefficients $\del\Gamma^k(\varphi_S,0)$ can be regarded as universal constants. Remark that in $S_W$ and $S_{KG}^k$, the cubic terms are linear combinations of
\begin{equation}\label{eq10-07-10-2020}
\aligned
&\phit^k \m^{\alpha\beta}\del_{\alpha}\phi^1\del_{\beta}\phi^1,\quad 
\phit^k \m^{\alpha\beta} \del_{\alpha}\phi^1\del_{\beta}\phit^j,\quad
\phit^k \del\phit^j\del \phit^l,
\\
&\phit^k\phit^j\del\phi^1,\quad
\phit^k\phit^j\del \phit^l,
\\
&\phit^j\phit^k\phit^l
\endaligned
\end{equation}
where $\m$ is the Minkowski metric. So the first two terms enjoy a null structure. The rest contains at least two Klein-Gordon factors.  We make the following estimates. First, by \eqref{eq13-10-06-2020} and the null condition of $\m^{\alpha\beta}, $
\begin{equation}\label{eq14-07-10-2020}
\aligned
|\m^{\alpha\beta}\del_{\alpha}\phi^1\del_{\beta}\phi^1|_{N-3}
\leq& C(s/t)^2|\del \phi^1|_{N-3}^2 + C|\dels \phi^1|_{N-3}|\del \phi^1|_{N-3}
\\
\leq& C(C_1\vep)^2(s/t)^2s^{-2+2\delta}
\endaligned
\end{equation}
where \eqref{eq2-07-10-2020} and \eqref{eq11-08-10-2020} are applied.
$$
\aligned
&\||\m^{\alpha\beta}\del_{\alpha}\phi^1\del_{\beta}\phi^1|_N\|_{L^2(\Hcal_s)}
\\
\leq& C\|(s/t)^2|\del \phi^1|_{N-3}|\del\phi^1|_N\|_{L^2(\Hcal_s)}
\\
    &+ C\||\dels \phi^1|_{N-3}|\del \phi^1|_{N}\|_{L^2(\Hcal_s)}
     + C\||\dels \phi^1|_{N}|\del \phi^1|_{N-3}\|_{L^2(\Hcal_s)}
\\
\leq& CC_1\vep s^{-1+\delta}\|(s/t)^2|\del \phi^1|_N\|_{L^2(\Hcal_s)}
\\ 
    &+ CC_1\vep s^{-1+\delta}\|(s/t)^2|\del \phi^1|_N\|_{L^2(\Hcal_s)}
     + CC_1\vep s^{-1+\delta}\||\dels \phi^1|_N\|_{L^2(\Hcal_s)}
\endaligned
$$
where \eqref{eq2-07-10-2020} and \eqref{eq11-08-10-2020} are applied. Then we conclude that
\begin{equation}\label{eq11-07-10-2020}
\||\m^{\alpha\beta}\del_{\alpha}\phi^1\del_{\beta}\phi^1|\|_{L^2(\Hcal_s)}\leq C(C_1\vep)^2s^{-1+2\delta}.
\end{equation}
Similarly,
\begin{equation}\label{eq12-07-10-2020}
|\m^{\alpha\beta}\del_{\alpha}\phi^1\del_{\beta}\phit^j|_{N-3}\leq C(C_1\vep)^2 (s/t)^2s^{-2+2\delta},
\end{equation}
\begin{equation}\label{eq13-07-10-2020}
\||\m^{\alpha\beta}\del_{\alpha}\phi^1\del_{\beta}\phit^j|_N\|_{L^2(\Hcal_s)}\leq C(C_1\vep)^2s^{-1+2\delta}.
\end{equation}

Then the first term in \eqref{eq10-07-10-2020} is bounded as following:
\begin{equation}\label{eq19-07-10-2020}
\aligned
&\|(s/t)^{-1}|\phit^k \m^{\alpha\beta}\del_{\alpha}\phi^1\del_{\beta}\phi^1|_N\|_{L^2(\Hcal_s)}
\\
\leq& \|(s/t)^{-1}|\phit^k|_{N-2}|\m^{\alpha\beta}\del_{\alpha}\phi^1\del_{\beta}\phi^1|_N\|_{L^2(\Hcal_s)}
\\
 &+ \|(s/t)^{-1}|\phit^k|_N|\m^{\alpha\beta}\del_{\alpha}\phi^1\del_{\beta}\phi^1|_{N-3}\|_{L^2(\Hcal_s)}
\\
\leq& CC_1\vep s^{-1+\delta}\||\m^{\alpha\beta}\del_{\alpha}\phi^1\del_{\beta}\phi^1|_N\|_{L^2(\Hcal_S)}
+ C(C_1\vep)^2 s^{-2+2\delta}\||\phit^k|_N\|_{L^2(\Hcal_s)}
\\
\leq& C(C_1\vep)^3 s^{-2+3\delta}
\endaligned
\end{equation}
where \eqref{eq14-07-10-2020} and \eqref{eq11-07-10-2020} are applied.

The second term in \eqref{eq10-07-10-2020} is bounded similarly with \eqref{eq12-07-10-2020} and \eqref{eq13-07-10-2020}:
\begin{equation}\label{eq20-07-10-2020}
\aligned
\|(s/t)^{-1}|\phit^k\m^{\alpha\beta}\del_{\alpha}\phit^j\del_{\beta}\phi^1|_N\|_{L^2(\Hcal_s)}
\leq&C(C_1\vep)^3s^{-2+3\delta}.
\endaligned
\end{equation}

The rest terms in \eqref{eq10-07-10-2020} contain at least two Klein-Gordon factors, which permits us to obtain sufficient $L^2$ bounds. We only write the bound on $\phit^k\phit^j\del\phi^1$ (which is the most critical one) and omit the rests.
\begin{equation}\label{eq17-07-10-2020}
\aligned
&\|(s/t)^{-1}|\phit^k\phit^j\del\phit^1|_N \|_{L^2(\Hcal_s)}
\\
\leq& \|(s/t)^{-1}|\phit^k|_{N-2}|\phit^j|_{N-2}|\del\phi^1|_N\|_{L^2(\Hcal_s)} 
\\
&+ \|(s/t)^{-1}|\phit^k|_{N-2}|\phit^j|_N|\del\phi^1|_{N-2}\|_{L^2(\Hcal_s)}
\\
    &+ \|(s/t)^{-1}|\phit^k|_N|\phit^j|_{N-2}|\del\phi^1|_{N-2}\|_{L^2(\Hcal_s)}
\\
\leq& C(C_1\vep)^2s^{-2+2\delta}\|(s/t)|\del\phi^1|_N\|_{L^2(\Hcal_s)}
     +C(C_1\vep)^2s^{-2+2\delta}\||\phit^j|_N\|_{L^2(\Hcal_s)}
\\
    &+C(C_1\vep)^2s^{-2+2\delta}\||\phit^j|_N\|_{L^2(\Hcal_s)}
\\
\leq& C(C_1\vep)^3 s^{-2+3\delta}.
\endaligned
\end{equation}

For forth and higher order terms, recall \eqref{eq16-07-10-2020}. There is at least two Klein-Gordon factor. So they are bounded similar to \eqref{eq17-07-10-2020}, we omit the detail.

Then, summarize \eqref{eq19-07-10-2020}, \eqref{eq20-07-10-2020}, \eqref{eq17-07-10-2020} and the above discussion, \eqref{eq1-06-10-2020} is concluded.

\paragraph*{Pointwise bounds on higher-order terms.}
We establish the following bounds:
\begin{equation}\label{eq12-08-10-2020}
|S_W[\phi^1,\phit^k]|_{N-3} + |S_{KG}^k[\phi^1,\phit^k]|_{N-3}\leq C(C_1\vep)^3(s/t)^2s^{-3+3\delta}.
\end{equation}
This also relies on \eqref{eq10-07-10-2020}. The first two null cubic forms are bounded via \eqref{eq14-07-10-2020} and \eqref{eq12-07-10-2020} combined with \eqref{eq9-05-10-2020}. The rest terms together with forth and higher-order terms, containing at least two Klein-Gordon factor  (among these the worst is $\phit^k\phit^j\del\phi^1$), are bounded directly by applying \eqref{eq9-05-10-2020}, \eqref{eq2-07-10-2020} and \eqref{eq11-08-10-2020}. 

\paragraph*{Improving the energy bounds on $w_0$.}
We apply directly Proposition \ref{prop 1 energy} on 
\begin{equation}\label{eq5-18-10-2020}
\Box \del^IL^J w_0 = \del^IL^J(S_W[\phi^1,\phit^k])
\end{equation}
for $|I|+|J|\leq N$. By \eqref{ineq 3 prop 1 energy}, we obtain, thanks to \eqref{eq1-06-10-2020},
$$
E_0(2,\del^IL^J w_0)^{1/2}\leq E_0(s,\del^IL^J w_0)^{1/2} + C(C_1\vep)^3\int_2^s\tau^{-2+3\delta}d\tau
$$
which leads to,
\begin{equation}
E_0(2,\del^IL^J w_0)^{1/2}\leq E_0(s,\del^IL^J w_0)^{1/2} + C(C_1\vep)^3. 
\end{equation}
Then we conclude that
\begin{equation}\label{eq26-07-10-2020}
\Ecal_0(s,w_0)^{1/2}\leq C_0\vep + C(C_1\vep)^3.
\end{equation} 

\paragraph*{Bounds on Hessian forms of $w_0$.}
Similar to the component $w$, we will establish:
\begin{equation}\label{eq5-09-10-2020}
(s/t)^2|\del\del w_0|_{N-3}\leq CC_1\vep (s/t)s^{-2}.
\end{equation}
This is also relied on Proposition \ref{prop1-14-08-2020}. Recall \eqref{eq12-08-10-2020} and \eqref{eq2 lem Hessian-flat-zero}, one has
$$
\aligned
(s/t)^2|\del\del w_0|_{N-3}\leq& Ct^{-1}|\del w_0|_{N-2} + C|\Box w_0|_{N-3}
\\
\leq& CC_1\vep(s/t)s^{-2} + C(C_1\vep)^3 (s/t)^2s^{-3+3\delta}.
\endaligned
$$
which leads to \eqref{eq5-09-10-2020}.

\paragraph*{Conformal energy  bound on $w_0$}
\begin{equation}\label{eq2-18-10-2020}
\Ecal_2^N(s,w_0)^{1/2}\leq CC_0\vep + C(C_1\vep)^3s^{3\delta}.
\end{equation}
We only need to apply Proposition \ref{prop-conformal} on \eqref{eq5-18-10-2020} for $|I|+|J|\leq N$. Recall \eqref{eq1-06-10-2020}, \eqref{eq2-18-10-2020} is concluded. 

Recalling \eqref{eq2-10-06-2020} and \eqref{eq4-10-06-2020}, we obtain the following bounds:
\begin{equation}\label{eq3-18-10-2020}
\|(s/t)^2s|\del w_0|_N\|_{L^2(\Hcal_s)}\leq CC_1\vep s^{3\delta},
\end{equation} 
\begin{equation}\label{eq4-18-10-2020}
|\del u|_{N-2}\leq CC_1\vep (s/t)^{-1}s^{-2+3\delta}.
\end{equation}

\subsection{Conical decay of Klein-Gordon component}
Parallel to \eqref{eq16-15-08-2020} and \eqref{eq17-15-08-2020}, we establish the following two bounds on $\phit^k$:
\begin{equation}\label{eq25-07-10-2020}
|\phit^k|_{N-3}\leq CC_1\vep (s/t)^2s^{-1+\delta},
\end{equation}
\begin{equation}\label{eq13-08-10-2020}
\|(s/t)^{-1}|\phit^k|_{N-1}\|_{L^2(\Hcal_s)}\leq CC_1\vep s^{\delta}.
 \end{equation}
To do se we apply Proposition \ref{prop1-fast-kg}. Then we need to bound the right-hand-side of the equation of $\phit^k$. The higher-order terms are bounded by \eqref{eq12-08-10-2020} and \eqref{eq1-06-10-2020}. The quadratic terms are bounded exactly as in Subsection \ref{subsec-model-conical}, because $(\phi^1,\phit^k)$ and $(w,v)$ satisfies the same bounds respectively. Then we conclude that
\begin{equation}\label{eq17-08-10-2020}
|\Box \phit^k + c^2\phit^k|_{N-3}\leq CC_1\vep |\phit^k|_{N-3} + C(C_1\vep)^3(s/t)^2s^{-3+3\delta},
\end{equation}
\begin{equation}\label{eq18-08-10-2020}
\|(s/t)^{-1}|\Box \phit^k + c^2\phit^k|_{N-1}\|_{L^2(\Hcal_s)}\leq  C(C_1\vep)^2s^{\delta} + CC_1\vep\|(s/t)^{-1}|\phit^k|_{N-1}\|_{L^2(\Hcal_s)}.
\end{equation}
Then following the argument in Subsection \ref{subsec-model-conical}, \eqref{eq25-07-10-2020} and \eqref{eq13-08-10-2020} are established. Here we also need the smallness condition on $C_1\vep$ as \eqref{eq7-08-10-2020}:
\begin{equation}\label{eq20-09-10-2020}
CC_1\vep\leq \frac{c^2}{2}.
\end{equation}
  
\subsection{Lower order energy bounds on Klein-Gordon components}
This is parallel to Subsection \ref{subsec-model-KG-lower}. We establish 
\begin{equation}\label{eq14-08-10-2020}
\sum_{k=2}^n\Ecal_{0,1}^{N-1}(s,\phit^k)^{1/2}\leq C_0\vep + C(C_1\vep)^2.
\end{equation}
The higher-order terms $S_{KG}^k[\phi]$ are bounded by \eqref{eq1-06-10-2020}. Furthermore, 
$$
\aligned
&\||\phit^j\del w_0|_{N-1}\|_{L^2(\Hcal_s)}
\\
\leq& CC_1\vep s^{-2+3\delta}\|(s/t)^{-1}|\phit^j|_{N-1}\|_{L^2(\Hcal_s)}
+ CC_1\vep s^{-2+2\delta}\|(s/t)^2s|\del w_0|_{N-1}\|_{L^2(\Hcal_s)}
\\
\leq& C(C_1\vep)^2 s^{-2+4\delta}
\endaligned
$$
where \eqref{eq4-18-10-2020}, \eqref{eq25-07-10-2020} are applied for the first inequality and \eqref{eq3-18-10-2020}, \eqref{eq13-08-10-2020} are applied for the second.
$$
\aligned
&\||\phit^j\del\del w|_{N-1}\|_{L^2(\Hcal_s)}
\\
\leq& 
C\||\phit^j|_{N-1}|\del\del w|_{N-3}\|_{L^2(\Hcal_s)} + C\||\del\del w|_{N-1}|\phit^j|_{N-3}\|_{L^2(\Hcal_s)}
\\
\leq& CC_1\vep s^{-2+2\delta}\|(s/t)^{-1}|\phit^j|_{N-1}\|_{L^2(\Hcal_s)}
+ CC_1\vep s^{-2+\delta}\|(s/t)^2s|\del\del w|_{N-1}\|_{L^2(\Hcal_s)}
\\
\leq&C(C_1\vep)^2 s^{-2+3\delta}.
\endaligned
$$
Here we have applied \eqref{eq5-07-10-2020} and \eqref{eq25-07-10-2020} for the second inequality, and \eqref{eq13-08-10-2020}, \eqref{eq6-07-10-2020} for the third inequality. These bounds are integrable, so we conclude by \eqref{eq14-08-10-2020}.

A direct result of \eqref{eq14-08-10-2020} is the following sharp bound in $\phit$:
\begin{equation}\label{eq16-08-10-2020}
s|\del \phit^k|_{N-3} + t|\phit^k|_{N-3} \leq CC_1\vep. 
\end{equation}
\subsection{Sharp decay bounds}
Now we are ready to establish the following sharp bounds:
\begin{equation}\label{eq15b-08-10-2020}
|\del\del w|_{N-4}\leq CC_1\vep s^{-1}.
\end{equation}

The proof is quite similar to that of \eqref{eq18-01-09-2020}. We remark that following the notation in Proposition \ref{prpo2 wave-sharp}, for $|I|+|J|\leq N-4$, 
$$
|S^w[\del_{\alpha}\del^IL^J u]|\leq CC_1\vep st^{-2}|\del w|_{N-2}
$$
which leads to
\begin{equation}\label{eq8-18-10-2020}
|\Delta^w[\del_{\alpha}\del^IL^J u]|\leq CC_1\vep t^{-2+\delta}.
\end{equation}
Furthermore, for $S^w[\del_{\alpha}\del^IL^J w]$, we need the following bound on $|\phit^k|$:
\begin{equation}\label{eq6-18-10-2020}
|\phit^k|_{N-4} \leq CC_1\vep (s/t)^2s^{-1}.
\end{equation}
This is proved as following. Recall Proposition \ref{prop1-fast-kg}, \eqref{eq16-08-10-2020} and \eqref{eq17-08-10-2020},
$$
\aligned
c^2|\phit^k|_{N-4}\leq& C(s/t)^2|\del \phit^k|_{N-3} + C|\Box \phit^k + c^2\phit^k|_{N-4}
\\
\leq& CC_1\vep(s/t)^2s^{-1} + CC_1\vep|\phit^k|_{N-4} + C(C_1\vep)^3(s/t)^2s^{-3+3\delta}.
\endaligned
$$
when $CC_1\vep\leq \frac{c^2}{2}$, \eqref{eq6-18-10-2020} is concluded. 

Then recall \eqref{eq16-08-10-2020} and above bound \eqref{eq6-18-10-2020},
\begin{equation}\label{eq7-18-10-2020}
|S^w[\del_{\alpha}\del^IL^J w]|\leq C|\phit^k|_{N-4}|\del \phit^k|_{N-4}\leq C|\phit^k|_{N-4}|\phit^k|_{N-3}\leq C(C_1\vep)^2(s/t)^2t^{-1}.
\end{equation}

Now we apply \eqref{eq1-29-05-2020} on 
$$
\Box \del^IL^J w = -\sum_{k=2}^n\del^IL^J \big(|\phit^k|^2\big).
$$
Substitute the above bounds \eqref{eq8-18-10-2020}, \eqref{eq7-18-10-2020} into \eqref{eq1-29-05-2020} and consider a point $(\bar{t},\bar{x})\in \Hcal_{\bar{s}}$,
$$
\aligned
&\bar{s}|\del_t\del_{\alpha} w|(\bar{t},\bar{x})
\\
\leq& CC_0\vep 
+ C(C_1\vep)^2\int_2^{\bar{t}} (s/t)^2 t^{-1}\big|_{\gamma(t;\bar{t},\bar{x})} e^{-\int_{t}^{\bar{t}}P_{\bar{t},\bar{x}}(\eta)d\eta} + CC_1\vep \int_2^{\bar{t}} t^{-2+\delta}dt
\\
\leq& CC_1\vep + C(C_1\vep)^2\int_2^{\bar{t}}(s/t)^2t^{-1}\big|_{\gamma(t;\bar{t},\bar{x})}e^{-\int_{t}^{\bar{t}}(s/t)^2t^{-1}\big|_{\gamma(\eta;\bar{t},\bar{x})}d\eta}dt
\\
\leq& CC_1\vep. 
\endaligned
$$
So we conclude that
\begin{equation}\label{eq9-18-10-2020}
|\del_t\del_{\alpha}\del^IL^J w|\leq CC_1\vep s^{-1}.
\end{equation}
Similar to the argument applied for \eqref{eq13-01-09-2020}, we conclude by \eqref{eq15b-08-10-2020}.

%
%


\subsection{Improved energy bounds and conclusion}
For \eqref{eq7-05-10-2020}, we follow a similar argument in Subsection \ref{subsec-model-improved}. Thanks to \eqref{eq16-08-10-2020},
\begin{equation}\label{eq9-09-10-2020}
\aligned
&\Ecal_0^p(s,w)^{1/2} + \sum_{\alpha}\Ecal_0^p(s,\del_{\alpha} w)^{1/2} 
\\
\leq& \Ecal_0^p(2,w)^{1/2} + \sum_{\alpha}\Ecal_0^p(2,\del_{\alpha} w)^{1/2}
+ CC_1\vep \sum_{k=2}^n\int_2^s\tau^{-1}\Ecal_{0,1}^p(\tau, \phit^k)^{1/2}d\tau.
\endaligned
\end{equation}
This is due to the following bound combined with Proposition \ref{prop 1 energy}:
$$
\||(\phit^k)^2|_p\|_{L^2(\Hcal_s)} + \||(\phit^k\del\phit^k)|_p\|_{L^2(\Hcal_s)}
\leq CC_1\vep s^{-1}\sum_{k=2}^n\Ecal_{0,1}^p(s,\phit^k)^{1/2}.
$$

The bounds on $\phit^k$ is similar. First, one has integrable $L^2$ bounds \eqref{eq1-06-10-2020} on higher-order terms $S_{KG}^k[\phi]$. Second, the term $vw_0$ does not appear. So we make the following bounds:
\begin{equation}\label{eq13-09-10-2020}
\||\phit^k\del w_0|_p\|_{L^2(\Hcal_s)}\leq CC_1\vep s^{-1}\sum_{k=2}^n\Ecal_{0,1}^p(s,\phit^k)^{1/2} + C(C_1\vep)^2s^{-2+4\delta}.
\end{equation} 
\begin{equation}\label{eq14-09-10-2020}
\aligned
\||\phit^k\del\del w|_p\|_{L^2(\Hcal_s)}\leq& CC_1\vep s^{-1}\Big(\sum_{k=2}^n\Ecal_{0,1}^p(s,\phit^k)^{1/2} 
+ \sum_{\alpha}\Ecal_0^p(s,\del_{\alpha}w)^{1/2}\Big).
\endaligned
\end{equation}

For \eqref{eq13-09-10-2020}, remark that
$$
\aligned
\||\phit^k\del w_0|_p\|_{L^2(\Hcal_s)}\leq& C\||\phit^k|_p|\del w_0|_{N-2}\|_{L^2(\Hcal_s)} + C\||\phit^k|_{N-4}|\del w_0|_p\|_{L^2(\Hcal_s)}
\\
\leq& CC_1\vep s^{-1}\||\phit^k|_p\|_{L^2(\Hcal_s)} 
+ CC_1\vep s^{-2 + \delta}\|(s/t)^2s|\del w_0|_p\|_{L^2(\Hcal_s)}
\\
\leq&CC_1\vep s^{-1} \sum_{k=2}^n\Ecal_{0,1}^p(s,\phit^k)^{1/2} + C(C_1\vep)^2 s^{-2 + 4\delta}
\endaligned
$$
where \eqref{eq1-07-10-2020}, \eqref{eq25-07-10-2020} and \eqref{eq3-18-10-2020} are applied.

For \eqref{eq14-09-10-2020}
$$
\aligned
&\||\phit^k\del\del w|_p\|_{L^2(\Hcal_s)}
\\
\leq& C\||\phit^k|_p|\del\del w|_{N-4}\|_{L^2(\Hcal_s)} 
+ C\||\phit^k|_{N-3}|\del\del w|_p\|_{L^2(\Hcal_s)}
\\
\leq& CC_1\vep s^{-1}\||\phit^k|_p\|_{L^2(\Hcal_s)} 
  + CC_1\vep s^{-1}\|(s/t)|\del\del w|_p\|_{L^2(\Hcal_s)}
\\
\leq& CC_1\vep s^{-1}\Big(\sum_{k=2}^n\Ecal_{0,1}^p(s,\phit^k)^{1/2} 
+ \sum_{\alpha}\Ecal_0^p(s,\del_{\alpha}w)^{1/2}\Big).
\endaligned
$$
Here for the second inequality, \eqref{eq16-08-10-2020} and \eqref{eq15b-08-10-2020} are applied. 

Now recall \eqref{eq1-06-10-2020} and substitute all these $L^2$ bounds into Proposition \ref{prop 1 energy}, we obtain:
\begin{equation}
\aligned
\sum_{k=2}^n\Ecal_{0,1}(s,\phit^k)^{1/2}\leq& \sum_{k=2}^n\Ecal_{0,1}(2,\phit^k)^{1/2} 
+ C(C_1\vep)^2 
\\
& + CC_1\vep \int_2^s\tau^{-1}\Big(\sum_{k=2}^n\Ecal_{0,1}^p(\tau,\phit^k)^{1/2} + \sum_{\alpha}\Ecal_0^p(s,\del_{\alpha}w)^{1/2}\Big) d\tau.
\endaligned
\end{equation}

Again, let 
$$
A^p(s) := \max\Big\{\sum_\alpha\Ecal_0^p(s,\del_{\alpha}w)^{1/2}, \Ecal_0^p(s,w)^{1/2},\sum_{k=2}^n\Ecal_{0,1}^p(s,\phit^k)^{1/2}\Big\}.
$$
Then for $0\leq p\leq N$,
\begin{equation}\label{eq16-09-10-2020}
\aligned
A^p(s)\leq& C_0\vep + C(C_1\vep)^2
 + CC_1\vep\int_2^s\tau^{-1}A^p(\tau)d\tau 
\endaligned
\end{equation}
So we conclude, thanks to Gronwall's inequality,  the energy bounds by 
\begin{equation}\label{eq19-09-10-2020}
A^{N}(s)\leq \big(C_0\vep  + C(C_1\vep)^2\big)s^{CC_1\vep}.
\end{equation}
\subsection{Conclusion of the bootstrap argument}
Now we are ready to improve the bootstrap bounds. \eqref{eq6-05-10-2020} is improved by \eqref{eq26-07-10-2020}. More precisely, if we take 
\begin{equation}\label{eq22a-09-10-2020}
C_1\geq 2C_0,\quad \vep\leq \sqrt{\frac{C_1-2C_0}{2CC_1^3}},
\end{equation}
then \eqref{eq26-07-10-2020} leads to \eqref{eq8-05-10-2020}. Furthermore, taking
\begin{equation}\label{eq22b-09-10-2020}
\vep\leq \frac{c^2}{2CC_1}
\end{equation}
in order to guarantee \eqref{eq20-09-10-2020}. Then taking
\begin{equation}\label{eq22c-09-10-2020}
\vep\leq \frac{C_1 - 2C_0}{2CC_1^{3/2}}, \quad \vep \leq \delta/CC_1,
\end{equation}
\eqref{eq7-05-10-2020} is guaranteed. Then taking $\vep_0$ to be the minimum of the above five quantity, the desired stability result is established.
\appendix

\bibliographystyle{elsarticle-num}
\bibliography{WKGm-bibtex}

\begin{thebibliography}{10}
\expandafter\ifx\csname url\endcsname\relax
  \def\url#1{\texttt{#1}}\fi
\expandafter\ifx\csname urlprefix\endcsname\relax\def\urlprefix{URL }\fi
\expandafter\ifx\csname href\endcsname\relax
  \def\href#1#2{#2} \def\path#1{#1}\fi

\bibitem{M-2020-strong}
Y.~Ma, Global solutions of nonlinear wave-klein-gordon system in two spatial
  dimensions: A prototype of strong coupling case, preprint arXiv:2008.10023v3.

\bibitem{Ab-2019}
L.~Abbrescia, Y.~Chen, Global stability of some totally geodesic wave maps,
  arXiv:1907.07226 [math.AP].

\bibitem{SS98}
J.~Shatah, M.~Struwe, Geometric wave equations, Vol. Courant Lecture Notes in
  Mathematics, vol. 2, American Mathematical Society, Providence, RI, 1998.

\bibitem{Kri07}
J.~Krieger, Global regularity and singularity development for wave maps,
  Surveys in differential geometry 12~(1) (2007) 167--202.

\bibitem{Zakharov-1972}
V.~E. Zakharov, Collapse of langmuir waves, Sov. Phys. JETP 35~(5) (1972)
  980--914.

\bibitem{Ozawa-1995}
K.~T. T.~Ozawa, Y.~Tsutsumi, Normal form and global solutions for the {K}lein -
  {G}ordon - {Z}akharov equations, Anna. de l'I.H.P. section C, tome 12, N. 4
  (1995) 459--503.

\bibitem{Tsutaya-1996}
K.~Tsutaya, Global existence of small amplitude solutions for the
  klein-gordon-zakharov equations, Nonlinear anal. - Theor. 27 (1996)
  1373--1380.

\bibitem{LM1}
P.~LeFloch, Y.~Ma, The hyperboloidal foliation method, World Scientific, 2015.

\bibitem{Dong-2020-2}
S.~Dong, Asymptotic behavior of the solution to the klein-gordon-zakharov model
  in dimension two, arXiv:2006.04443v1 [math.AP].

\bibitem{M4}
Y.~Ma, Global solutions of nonlinear wave-{K}lein-gordon system in two spatial
  dimensions: weak coupling case, preprint arXiv:1907.03516.

\bibitem{Sogge-2008-book}
C.~Sogge, Lectures on Non-linear wave equations (2nd edition), International
  Press Boston, Inc., 2008.

\bibitem{Stingo-2018}
A.~Stingo, Global existence of small amplitude solutions for a model quadratic
  quasi-linear coupled wave-{K}lein-{G}ordon system in two space dimension,
  with mildly decaying cauchy data, arXiv:1507.02035v1.

\bibitem{Kl2}
S.~Klainerman, Global existence of small amplitude solutions to nonlinear
  {K}lein-{G}ordon equations in four-spacetime dimensions, Commun. Pure Appl.
  Math. 38~(1) (1985) 631--641.
\newblock \href {http://dx.doi.org/10.1002/cpa.3160380512}
  {\path{doi:10.1002/cpa.3160380512}}.

\bibitem{Dfx}
J.-M. Delort, {D. Fang and R. Xue}, Global existence of small solutions for
  quadratic quasilinear {K}lein-{G}ordon systems in two space dimensions, J.
  Funct. Anal. 211~(2) (2004) 288--323.
\newblock \href {http://dx.doi.org/10.1016/j.jfa.2004.01.008}
  {\path{doi:10.1016/j.jfa.2004.01.008}}.

\bibitem{KS-2011}
Y.~Kawahara, H.~Sunagawa, Global small amplitude solutions for two-dimensional
  nonlinear klein-gordon systems in the presence of mass resonance, J. Differ.
  Equations 251~(9) (2011) 2549--2567.
\newblock \href {http://dx.doi.org/10.1016/j.jde.2011.04.001}
  {\path{doi:10.1016/j.jde.2011.04.001}}.

\bibitem{M1}
Y.~Ma, Global solutions of quasilinear wave-{K}lein-{G}ordon system in two
  space dimension: technical tools, J. Hyperbol. Differ. Eq. 14~(4) (2017)
  591--625.
\newblock \href {http://dx.doi.org/10.1142/S0219891617500205}
  {\path{doi:10.1142/S0219891617500205}}.

\bibitem{ES64}
J.~Eells, J.-H. Sampson, Harmonic mappings of riemannian manifolds, Am. J.
  Math. 86~(1)  109--160.

\bibitem{Vil70}
J.~Vilms, Totally geodesic maps, J. Differ. Geom. 4~(1) (1970) 73--79.

\bibitem{A2}
S.~Alinhac, The null condition for quasilinear wave equations in two-space
  dimension {I}, Invent. math. 145~(3) (2001) 597--618.
\newblock \href {http://dx.doi.org/10.1007/s002220100165}
  {\path{doi:10.1007/s002220100165}}.

\end{thebibliography}
\end{document}